\documentclass[11pt,twoside,a4paper]{amsart}
\usepackage{amscd}

\usepackage{amsmath}
\usepackage{amssymb}

\input{xy}
\xyoption{all}

\newcommand{\FQSym}{\mathbf{FQSym}}
\newcommand{\WQSym}{\mathbf{WQSym}}
\newcommand{\QSh}{{\mathbf{QSh}}}
\newcommand{\QShi}{\mathit{QSh}}
\newcommand{\Sh}{\mathbf{Sh}}
\newcommand{\Shi}{\mathit{Sh}}

\newcommand{\NQShi}{\mathit{NQSh}}
\newcommand{\NSh}{{\mathbf{NSh}}}

\newcommand{\Com}{\mathbf{Com}}
\newcommand{\Comi}{\mathit{Com}}
\newcommand{\As}{{\mathbf{As}}}
\newcommand{\Asi}{\mathit{As}}
\newcommand{\Desc}{\mathbf{Desc}}
\newcommand{\QDesc}{\mathbf{QDesc}}
\newcommand{\PBT}{\mathbf{PBT}}
\newcommand{\ST}{\mathbf{ST}}

\renewcommand{\S}{\mathfrak{S}}

\newcommand{\tdelta}{\tilde{\Delta}}
\newcommand{\N}{\mathbb{N}}
\newcommand{\TSchroder}{\mathbb{T}_{Sch}}
\newcommand{\Tbin}{\mathbb{T}_{bin}}
\newcommand{\Surj}{\mathbb{S}urj}

\def\shuff#1#2{\mathbin{
      \hbox{\vbox{\hbox{\vrule \hskip#2 \vrule height#1 width 0pt}\hrule}\vbox{\hbox{\vrule \hskip#2 \vrule height#1 width 0pt\vrule }\hrule}}}}
\def\shuffl{{\mathchoice{\shuff{5pt}{3.5pt}}{\shuff{5pt}{3.5pt}}{\shuff{3pt}{2.6pt}}{\shuff{3pt}{2.6pt}}}}
\def\shuffle{\, \shuffl \,}

\def\qshuffl{{\mathchoice{\shuff{5pt}{3.5pt}\hspace{-2.9mm}-}{\shuff{5pt}{3.5pt}\hspace{-2.9mm}-}
{\shuff{3pt}{2.6pt}\hspace{-2.2mm}-}{\shuff{3pt}{2.6pt}\hspace{-2.2mm}-}}}
\def\qshuffle{\,\qshuffl\,}

\newtheorem{defi}{\indent Definition}
\newtheorem{lemma}[defi]{\indent Lemma}
\newtheorem{cor}[defi]{\indent Corollary}
\newtheorem{theo}[defi]{\indent Theorem}
\newtheorem{prop}[defi]{\indent Proposition}
\newtheorem{exam}[defi]{\indent Example}

\newcommand{\bun}{
\begin{picture}(5,5)(-2,0)
\put(0,0){\line(0,0){5}}
\end{picture}}

\newcommand{\bdeux}{
\begin{picture}(10,15)(-5,0)
\put(-3.5,4.1){$\vee$}
\put(0,0){\line(0,1){5}}
\end{picture}}

\newcommand{\btroisun}{
\begin{picture}(12,20)(-6,0)
\put(-3.5,4.1){$\vee$}
\put(0,0){\line(0,1){5}}
\put(-6.3,10.7){$\vee$}
\end{picture}}
\newcommand{\btroisdeux}{
\begin{picture}(12,20)(-6,0)
\put(-3.5,4.1){$\vee$}
\put(0,0){\line(0,1){5}}
\put(-0.7,10.7){$\vee$}
\end{picture}}
\newcommand{\schtrois}{
\begin{picture}(10,15)(-5,0)
\put(-3.5,4.1){$\vee$}
\put(0,0){\line(0,1){11}}
\end{picture}}

\newcommand{\bquatreun}{
\begin{picture}(15,25)(-7,0)
\put(-3.5,4.1){$\vee$}
\put(0,0){\line(0,1){5}}
\put(-6.3,10.7){$\vee$}
\put(-9.1,17.3){$\vee$}
\end{picture}}
\newcommand{\bquatredeux}{
\begin{picture}(15,25)(-7,0)
\put(-3.5,4.1){$\vee$}
\put(0,0){\line(0,1){5}}
\put(-6.3,10.7){$\vee$}
\put(-3.5,17.3){$\vee$}
\end{picture}}
\newcommand{\bquatretrois}{
\begin{picture}(15,25)(-7,0)
\put(-3.5,4.1){$\vee$}
\put(0,0){\line(0,1){5}}
\put(-0.7,10.7){$\vee$}
\put(-3.5,17.3){$\vee$}
\end{picture}}
\newcommand{\bquatrequatre}{
\begin{picture}(15,25)(-7,0)
\put(-3.5,4.1){$\vee$}
\put(0,0){\line(0,1){5}}
\put(-0.7,10.7){$\vee$}
\put(2.1,17.3){$\vee$}
\end{picture}}
\newcommand{\bquatrecinq}{
\begin{picture}(20,20)(-10,0)
\put(0,0){\line(0,1){5}}
\put(0,5){\line(2,1){10}}
\put(0,5){\line(-2,1){10}}
\put(-3,6.5){\line(0,1){5}}
\put(3,6.5){\line(0,1){5}}
\end{picture}}
\newcommand{\schquatreun}{
\begin{picture}(12,20)(-6,0)
\put(-3.5,4.1){$\vee$}
\put(0,0){\line(0,1){5}}
\put(-6.3,10.7){$\vee$}
\put(-2.7,11){\line(0,1){6}}
\end{picture}}
\newcommand{\schquatredeux}{
\begin{picture}(15,20)(-7,0)
\put(-3.5,4.1){$\vee$}
\put(0,0){\line(0,1){11}}
\put(-6.3,10.7){$\vee$}
\end{picture}}
\newcommand{\schquatretrois}{
\begin{picture}(15,20)(-7,0)
\put(-3.5,4.1){$\vee$}
\put(0,0){\line(0,1){11}}
\put(-3.5,10.7){$\vee$}
\end{picture}}
\newcommand{\schquatrequatre}{
\begin{picture}(15,20)(-7,0)
\put(-3.5,4.1){$\vee$}
\put(0,0){\line(0,1){11}}
\put(-0.7,10.7){$\vee$}
\end{picture}}
\newcommand{\schquatrecinq}{
\begin{picture}(15,20)(-7,0)
\put(-3.5,4.1){$\vee$}
\put(0,0){\line(0,1){5}}
\put(-0.7,10.7){$\vee$}
\put(3.1,11){\line(0,1){6}}
\end{picture}}
\newcommand{\schquatresix}{
\begin{picture}(20,20)(-10,0)
\put(0,0){\line(0,1){5}}
\put(0,5){\line(2,1){10}}
\put(0,5){\line(-2,1){10}}
\put(-3.5,4.4){$\vee$}
\end{picture}}

\title{Lie theory for quasi-shuffle bialgebras}
\date{}

\begin{document}

\author{Lo\"\i c Foissy}
\address{Laboratoire de Math\'ematiques\\
 Universit\'e de Reims\\
 Moulin de la Housse\\
B.P.1039\\  51687 Reims Cedex 2\\ France}
\email{loic.foissy@univ-reims.fr}

\author{Fr\'ed\'eric Patras}
\address{UMR 6621 CNRS\\
        		Universit\'e de Nice\\
        		Parc Valrose\\
        		06108 Nice Cedex 02\\
        		France}
        		\email{patras@unice.fr}

\maketitle

\section{Introduction}
Enveloping algebras of Lie algebras are known to be a fundamental notion, for an impressive variety of reasons. Their bialgebra structure allows to make a natural bridge between Lie algebras and groups. As such they are a key tool in pure algebra, algebraic and differential geometry, and so on. Their combinatorial structure is interesting on its own and is the object of the theory of free Lie algebras. Applications thereof include the theory of differential equations, numerics, control theory... 
From the modern point of view, featured in Reutenauer's \it Free Lie algebras \rm \cite{Reutenauer}, the ``right'' point of view on enveloping algebras is provided by the descent algebra: most of their key properties can indeed be obtained and finely described using computations in symmetric group algebras relying on the statistics of descents of permutations. More recently, finer structures have emerged that refine this approach. Let us quote, among others, the Malvenuto-Reutenauer 
Hopf algebra \cite{MR} and its bidendriform structure \cite{foissy2007bidendriform}.

Many features of classical Lie theory generalize to the broader context of algebras over Hopf operads \cite{Livernet}. However, this idea remains largely to be developed systematically. Quasi-shuffle algebras provide for example an interesting illustration of these phenomena, but have not been investigated from this point of view.

The notion of quasi-shuffle algebras can be traced back to the beginings of the theory of Rota--Baxter algebras, but was developed systematically only recently, starting essentially with Hoffman's work, that was motivated by multizeta values (MZVs) and featured their bialgebra structure.
Many partial results on the fine structure of quasi-shuffle bialgebras have been obtained since then \cite{Hoffman,Loday,Novelli,novelli2,foissy2}
but, besides the fact that each of these articles features a particular point of view, they fail to develop systematically a complete theory.

This article builds on these various results and develops the analog theory, for quasi-shuffle algebras, of the theory of descent algebras and their relations to free Lie algebras for classical enveloping algebras.

The plan is as follows. Sections \ref{sect:2} and \ref{sect:3} recall the fundamental definitions. These are fairly standard ideas and materials, excepted for the fact that bialgebraic structures are introduced from the point of view of Hopf operads that will guide later developments. The following section shows how the symmetrization process in the theory of twisted bialgebras (or Hopf species) can be adapted to define a noncommutative quasi-shuffle bialgebra structure on the operad of quasi-shuffle algebras. The properties of its primitive elements are studied from an operadic and enveloping algebra (left adjoint) point of view. Section \ref{sect:5} deals with the algebraic structure of linear endomorphisms of quasi-shuffle bialgebras and studies from this point of view the structure of surjections. Section \ref{sect:6} deals with the projection on the primitives of quasi-shuffle bialgebras -the analog in the present setting of the canonical projection from an enveloping algebra to the Lie algebra of primitives. As in classical Lie theory, a structure theorem for quasi-shuffle algebras follows from the properties of this canonical projection.
Section \ref{sect:7} investigates the relations between the shuffle and quasi-shuffle operads when both are equipped with the Hopf algebra structure inherited from the Hopf operadic structure of their categories of algebras (as such they are isomorphic respectively to the Malvenuto-Reutenauer Hopf algebra, or Hopf algebra of free quasi-symmetric functions, and to the Hopf algebra of word quasi-symmetric functions). We recover in particular the exponential isomorphism relating shuffle and quasi-shuffle bialgebras.
Section \ref{sect:8} studies coalgebra endomorphisms of quasi-shuffle bialgebras and classifies natural Hopf algebra endomorphisms and morphisms relating shuffle and quasi-shuffle bialgebras. 
Section \ref{sect:9} studies coderivations. Quasi-shuffle bialgebras are considered classically as filtered objects (the product does not respect the tensor graduation), however the existence of a natural graded Hopf algebra structure can be deduced from the general properties of their coderivations.
Section \ref{sect:10} explains briefly how the formalism of operads can be adapted to take into account graduations by using decorated operads. We detail then the case of quasi-shuffle algebras and conclude by initiating the study of the analogue, in this context, of the classical descent algebra.
Section \ref{sect:11} shows, using the bidendriform rigidity theorem, that the decorated quasi-shuffle operad is free as a noncommutative shuffle algebra. Section \ref{sect:12} shows that the quasi-shuffle analog of the descent algebra, $\QDesc$, is, up to a canonical isomorphism, a free noncommutative quasi-shuffle algebra over the integers. The last section concludes by investigating the quasi-shuffle analog of the classical sequence of inclusions $\Desc\subset \PBT\subset \Sh$ of the descent algebra into the algebra of planar binary trees, resp. the operad of shuffle algebras. In the quasi-shuffle context, this sequence reads $\Desc\subset \ST\subset \QSh$, where $\ST$ stands for the algebra of Schr\"oder trees and $\QSh$ for the quasi-shuffle operad.

\ \\
{\bf Notations and conventions}
All the structures in the article (vector spaces, algebras, tensor products...) are defined over a field $k$. Algebraic theories and their categories ($\Comi ,\Asi , \Shi ,\QShi \dots$) are denoted in italic, as well as the corresponding free algebras over sets or vector spaces ($QSh(X),Com(V)\dots$). Operads (of which we will study underlying algebra structures) and abbreviations of algebra names are written in bold ($\QSh ,\NSh ,\Com , \FQSym\dots$).

\ \\
{\bf Acknowledgements} {The authors were supported by the grant CARMA ANR-12-BS01-0017. We thank its participants and especially Jean-Christophe Novelli and Jean-Yves Thibon, for stimulating discussions on noncommutative symmetric functions and related structures. This article is, among others, a follow up of our joint works \cite{novelli2,foissy2}.
We also thank the ICMAT Madrid for its hospitality.

\section{Quasi-shuffle algebras}\label{sect:2}

Quasi-shuffle algebras have mostly their origin in the theory of Rota-Baxter algebras and related objects such as MZVs (this because the summation operator of series is an example of a Rota--Baxter operator). As we just mentioned, this is in often traced back to Cartier's construction of free commutative Rota-Baxter algebras \cite{Cartier} but the recent developments really started with Hoffman's \cite{Hoffman}. Further historical details and references can be found in the survey article \cite{KEFsmf}.

Another reason for the development of the theory lies in the theory of combinatorial Hopf algebras and, more specifically, into the developments originating in the theory of quasi-symmetric functions and the dual theory of noncommutative symmetric functions. This line of thought is illustrated in \cite{Novelli,novelli2,foissy2}.

Still another approach originates in the work of Chapoton on the combinatorial and operadic properties of permutohedra and other polytopes
(see e.g. \cite{chapoton2000algebres,chapoton2002operades} and the introduction of \cite{Novelli}). These phenomena lead to the axiomatic definition of tridendriform algebras (or dendriform trialgebras) in \cite{Loday}. 
Commutative tridendriform algebras (CTAs) identify with quasi-shuffle algebras (in the sense that the structure axioms of CTAs are the analog, for quasi-shuffle algebras, of Sch\"utzenberger's axioms for shuffle algebras \cite{schutz})
-we prefer the terminology quasi-shuffle, better established and more intuitive.

We follow here the Rota--Baxter approach, which is the one underlying at the moment most of the applications of the theory and the motivations for its development (besides MZVs and the works initiated by Hoffman in this area, one can mention the field of stochastic integration \cite{ebrahimi2015flows,ebrahimi2015exponential}).

\begin{defi}
A Rota--Baxter (RB) algebra of weight $\theta$ is an associative algebra $A$ equipped with a linear endomorphism $R$ such that
$$\forall x,y\in A, R(x)R(y)=R(R(x)y+xR(y)+\theta xy).$$
It is a commutative Rota--Baxter algebra if it is commutative as an algebra.
\end{defi}

Setting $R':=R/\theta$ when $\theta \not=0$, one gets that the pair $(A,R')$ is a Rota--Baxter algebra of weight 1. This implies that, in practice, there are only two interesting cases to be studied abstractly: the weight 0 and weight 1 (or equivalently any other non zero weight). The others can be deduced easily from the weight 1 case. The same observation applies for one-parameter variants of the notion of quasi-shuffle algebras.

From now on in this article, \it RB algebra \rm will stand for \it RB algebra of weight 1\rm . When other RB algebras will be considered, their weight will be mentioned explicitely.

An important property of RB algebras, whose proof is left to the reader, is the existence of an associative product, the RB double product $\star$, defined by:
\begin{equation}
x\star y:=R(x)y+xR(y)+xy
\end{equation}
so that: $R(x)R(y)=R(x\star y)$.
If one sets, in a RB algebra, $x\prec y:=xR(y),\ x\succ y:=R(x)y$, one gets immediately relations such as 
$$(x\cdot y)\prec z=xyR(z)=x\cdot (y\prec z),$$
$$ (x\prec y)\prec z=xR(y)R(z)=x\prec (y\star z),$$
and so on. In the commutative case, $x\prec y=y\succ x$, 
and all relations between the products $\prec ,\succ , \cdot$ 
and $\star :=\prec +\succ +\cdot$ follow from these two. 
In the noncommutative case, the relations duplicate and one has furthermore $(x\succ y)\prec z=R(x)yR(z)=x\succ (y\prec z).$
These observations give rise to the axioms of quasi-shuffle algebras (or CTAs) and noncommutative quasi-shuffle  algebras (NQSh or tridendriform algebras)\cite{Loday}.

From now on, ``commutative algebra'' without other precision means commutative and associative algebra; ``product'' on a vector space $A$ means a bilinear product, that is a linear map from $A\otimes A$ to $A$.

\begin{defi}
A quasi-shuffle (QSh) algebra $A$ is a nonunital commutative algebra (with product written $\bullet$) equipped with another product $\prec$ such that
\begin{align}
\label{QS1}(x\prec y)\prec z&=x\prec(y\star z)\\
\label{QS2}(x\bullet y)\prec z&=x\bullet(y\prec z).
\end{align}
where $x\star y:=x\prec y+y\prec x+x\bullet y$. We also set for further use $x\succ y:=y\prec x$.
As the RB double product in a commutative RB algebra, the product $\star$ is automatically associative and commutative and defines another commutative algebra structure on $A$. \end{defi}

Recall, for further use, that \it shuffle algebras \rm correspond to weight 0 commutative RB algebras, that is quasi-shuffle algebras with a null product $\bullet =0$. Equivalently:
\begin{defi} A shuffle (Sh) algebra is a vector space equipped with a product $\prec$ satisfying (\ref{QS1}) with $x\star y:=x\prec y+y\prec x$, see e.g. \cite{schutz,KEFsmf,Foissy} for further details.
\end{defi}

It is sometimes convenient to equip quasi-shuffle algebras with a unit. The phenomenon is exactly similar to the case of shuffle algebras \cite{schutz}: given a quasi-shuffle algebra, one sets $B:=k\oplus A$, and the products $\prec$, $\bullet$ have a partial extension to $B$ defined by, for $x\in A$:
$$1\bullet x=x\bullet 1:=0,  \ 1\prec x:=0,\ x\prec 1:=x.$$
The products $1\prec 1$ and $1\bullet 1$ cannot be defined consistenly, but one sets $1\star 1:=1$, making $B$ a unital commutative algebra for $\star$.

The categories of quasi-shuffle and of unital quasi-shuffle algebras are clearly equivalent (under the operation of adding or removing a copy of the ground field).

\begin{defi}
A noncommutative quasi-shuffle   algebra (NQSh or tridendriform algebra) is a nonunital associative algebra (with product written $\bullet$) equipped with two other products $\prec, \succ$
such that, for all $x,y,z\in A$:
\begin{align}
\label{E1}(x\prec y)\prec z&=x\prec(y\star z)\\
\label{E2}(x\succ y)\prec z&=x\succ(y\prec z)\\
\label{E3}(x\star y) \succ z&=x\succ (y\succ z)\\
\label{E5}(x\prec y)\bullet z&=x\bullet (y\succ z)\\
\label{E6}(x \succ y)\bullet z&=x\succ (y\bullet z)\\
\label{E7}(x\bullet y)\prec z&=x\bullet (y\prec z).
\end{align}
where $x\star y:=x\prec y+x\succ y+x\bullet y$.
 \end{defi}
As the RB double product, the product $\star$ is automatically associative and equips $A$ with another associative algebra structure.
Indeed, the associativity relation 
\begin{equation}\label{E4}
(x\bullet y)\bullet z= x\bullet (y\bullet z)
\end{equation} and $(\ref{E1})+\ldots+(\ref{E7})$ imply the associativity of $\star$:
\begin{equation}
(x\star y)\star z=x\star (y\star z).
\end{equation}
If $A$ is furthermore a quasi-shuffle algebra, then the product $\star$ is commutative.
One can show that these properties are equivalent to the associativity of the double product $\star$ in a Rota-Baxter algebra (this is because the free NQSh algebras embed into the corresponding free Rota--Baxter algebras).

\it Noncommutative shuffle algebras \rm correspond to weight 0 RB algebras, that is NQSh algebras with a null product $\bullet =0$. Equivalently:
\begin{defi}
A noncommutative shuffle (NSh or dendriform) algebra is a vector space equipped with two products $\prec,\succ$ satisfying (\ref{E1},\ref{E2},\ref{E3}) with $x\star y:=x\prec y+y\prec x$. 
\end{defi}

The most classical example of such a structure is provided by the topologists' shuffle product and its splitting into two ``half-shuffles'' 
\cite{EM}.

As in the commutative case, it is sometimes convenient to equip NQSh algebras with a unit. Given a NQSh algebra, one sets $B:=k\oplus A$, and the products $\prec$, $\succ$, $\bullet$ have a partial extension to $B$ defined by, for $x\in A$:
$$1\bullet x=x\bullet 1:=0,  \ 1\prec x:=0,\ x\prec 1:=x, \ 1\succ x:=x, \ x\succ 1:=0.$$
The products $1\prec 1$, $1\succ 1$ and $1\bullet 1$ cannot be defined consistenly, but one sets $1\ast 1:=1$, making $B$ a unital commutative algebra for $\ast$.

The categories of NQSh and unital NQSh algebras are clearly equivalent.

\ \par

The following Lemma encodes the previously described relations between RB algebras and quasi-shuffle algebras:
\begin{lemma}
The identities $x\prec y:=xR(y),\ x\succ y:=R(x)y, x\bullet y:=xy$ induce a forgetful functor from RB algebras to NQSh algebras, resp. from commutative RB algebras to QSh algebras.
\end{lemma}

{\bf Remarks.} Let $A$ be a NQSh algebra. 
\begin{enumerate}
\item If $A$ is a commutative algebra (for the product $\bullet$) and if for $x,y\in A$: $x\prec y=y\succ x,$ we say that $A$ is commutative as a NQSh algebra. Then, $(A,\bullet,\prec)$ is a quasi-shuffle algebra.
\item We put $\preceq=\prec+\bullet$. Then $(\ref{E1})+(\ref{E5})+(\ref{E7})+(\ref{E4})$, $(\ref{E2})+(\ref{E7})$ and $(\ref{E3})$
give: 
\begin{align}
(x \preceq y)\preceq z&=x\preceq (y\preceq z+y\succ z),\\
(x \succ y)\preceq y&=x\succ (y\preceq z),\\
(x\preceq y+x\succ y)\succ z&=x\succ (y\succ z).
\end{align}
These are the axioms that define a noncommutative shuffle algebra structure $(A,\preceq,\succ)$ on $A$. Similarly, if  $\succeq=\succ+\bullet$, then $(A,\prec,\succeq)$ is a noncommutative shuffle algebra.
\end{enumerate}

\begin{exam}[Hoffman, \cite{Hoffman}] Let $V$ be an associative, non unitary algebra. The product of $v,w\in V$ is denoted by $v.w$. The augmentation ideal $T^+(V)=\bigoplus\limits_{n\in\N^\ast}V^{\otimes n}$
of the tensor algebra $T(V)=\bigoplus\limits_{n\in\N^\ast}T_n(V)=\bigoplus\limits_{n\in\N^\ast}V^{\otimes n}$ (resp. $T(V)$) is given  a unique (resp. unital) NQSh algebra structure by induction on the length of tensors such that for all $a,b \in V$, for all $v,w \in T(V)$:
\begin{align}
av\prec bw&=a(v\qshuffle bw),&
av \succ bw&=b(av \qshuffle w),&
av \bullet bw&=(a.b)(v \qshuffle w),
\end{align}
where $\qshuffle=\prec+\succ+\bullet$ is called the quasi-shuffle product on $T(V)$ (by definition: $\forall v\in T(V), 1\qshuffle v=v=v\qshuffle 1$). 

\begin{defi}
The four-tuple $(T^+(V),\prec,\succ,\bullet)$ is the \emph{tensor quasi-shuffle algebra}
associated to $V$. It is a NQSh algebra and is a quasi-shuffle algebra if, and only if, $(V,.)$ is commutative (and then is called simply the \emph{quasi-shuffle} algebra associated to $V$).
\end{defi}

Here are examples of products in $T^+(V)$. Let $a,b,c \in V$.
\begin{align*}
a\prec b&=ab,&a\succ b&=ba,&a\bullet b&=a.b,\\
a\prec bc&=abc,&a\succ bc&=bac+bca+b(a.c),&a\bullet bc&=(a.b)c,\\
ab \prec c&=abc+acb+a(b.c),&ab \succ c&=cab,&ab \bullet c&=(a.c)b. 
\end{align*}
In particular, the restriction of $\bullet$ to $V$ is the product of $V$. If the product of $V$ is zero, we obtain the usual shuffle product $\shuffle$.

A very useful observation, to which we will refer as "Sch\"utzenberger's trick" (see \cite{schutz}) is that, in $T^+(V)$, for $v_1,\dots,v_n\in V$,
\begin{equation}\label{schutztrick}
v_1\dots v_n=v_1\prec (v_2\prec \dots (v_{n-1}\prec v_n)\dots )).
\end{equation}
\end{exam}

\section{Quasi-shuffle bialgebras}\label{sect:3}
We recall that graded connected and more generally conilpotent bialgebras are automatically equipped with an antipode \cite{Cartier2}, so that the two notions of bialgebras and Hopf algebras identify when these conditions are satisfied --this will be most often the case in the present article.

Quasi-shuffle bialgebras are particular deformations of shuffle bialgebras associated to the exponential and logarithm maps. They were first introduced by Hoffman in \cite{Hoffman} and studied further in \cite{loday3,foissy2}. The existence of a natural isomorphism between the two categories of bialgebras is known as Hoffman's isomorphism \cite{Hoffman} and has been studied in depth in \cite{foissy2}.

We introduce here a theoretical approach to their definition, namely through the categorical notion of  Hopf operads, see \cite{Livernet}. 
The underlying ideas are elementary and deserve probably to be better known. We avoid using the categorical or operadic langage and present them simply (abstract definitions and further references on the subject are given in \cite{Livernet}).

Let us consider categories of binary algebras, that is algebras defined by one or several binary products satisfying homogeneous multilinear relations (i.e. algebras over binary operads). For example, commutative algebras are algebras equipped with a binary product $\cdot$ satisfying the relations $x\cdot (y\cdot z)=(x\cdot y)\cdot z$ and $x\cdot y=y\cdot x$, and so on. Multilinear means that letters should not be repeated in the defining relations: for example, $n$-nilpotent algebras defined by a binary product with $x^n=0,\ n>1$ are excluded.

The category of algebras will be said \it non-symmetric \rm if in the defining relations the letters $x,y,z...$ always appear in the same order. For example, the category ${\Comi}$ of commutative algebras is not non-symmetric because of the relation $x\cdot y=y\cdot x$, whereas ${\Asi}$, the one of associative algebras  ($x\cdot (y\cdot z)=(x\cdot y)\cdot z$) is. 

Notice that the categories $\mathit{Sh,\ QSh}$ of shuffle and quasi-shuffle algebras are not non-symmetric (respectively because of the relation $x\star y=x\prec y+y\prec x$ and because of the commutativity of the $\bullet$ product) and are equipped with a forgetful functor to ${\Comi}$. The categories $\mathit{NSh,\ NQSh}$ of noncommutative shuffle and quasi-shuffle algebras are non-symmetric (in their defining relations the letters $x,y,z$ are not permuted) and are equipped with a forgetful functor to ${\Asi}$. 

\begin{defi}
Let ${\mathit{C}}$ be a category of binary algebras. The category is said {\rm Hopfian} if tensor products of algebras in $\mathit C$ are naturally equipped with the structure of an algebra in $\mathit C$ (i.e. the tensor product can be defined internally to $\mathit C$). 
\end{defi}
Classical examples of Hopfian categories are ${\Comi}$ and ${\Asi}$.
\begin{defi}
A bialgebra in a Hopfian category of algebras ${\mathit{C}}$ (or $\mathit C$-bialgebra) is an algebra $A$ in ${\mathit{C}}$ equipped with a coassociative morphism to $A\otimes A$ in ${\mathit{C}}$.
\end{defi}
Equivalently, it is a coalgebra in the tensor category of ${\mathit{C}}$-algebras.

Further requirements can be made in the definition of bialgebras, for example when algebras have units. When ${\mathit{C}}={\Comi}$ or ${\Asi}$, we recover the usual definition of bialgebras.

\begin{prop}
A category of binary algebras equipped with a forgetful functor to $\mathit{Com}$ is Hopfian. In particular, $\mathit{Pois,Sh,QSh}$ are Hopfian.
\end{prop}

Here ${\mathit{Pois}}$ stands for the category of Poisson algebras, studied in \cite{Livernet} from this point of view.

Indeed, let ${\mathit{C}}$ be a category of binary algebras equipped with a forgetful functor to ${\Comi}$. We write $\mu_1,\dots,\mu_n$ the various binary products on $A,B\in{\mathit{C}}$ and $\cdot$ the commutative product (which may be one of the $\mu_i$, or be induced by these products as the $\star$ product is induced by the $\prec, \succ$ and $\bullet$ products in the case of shuffle and quasi-shuffle algebras). Notice that a given category may be equipped with several distinct forgetful functors to ${\Comi}$: the quasi-shuffle algebras carry, for example, two commutative products ($\bullet$ and $\star$).

The Proposition follows by defining properly the $\mathit C$-algebra structure on the tensor products $A\otimes B$:
$$\mu_i(a\otimes b,a'\otimes b'):=\mu_i(a,a')\otimes b\cdot b'.$$
The new products $\mu_i$ on $A\otimes B$ clearly satisfy the same relations as the corresponding products on $A$, which concludes the proof. Notice that one could also define a ``right-sided'' structure by $\mu_i(a\otimes b,a'\otimes b'):=a\cdot a'\otimes \mu_i(b, b').$

A  bialgebra (without a unit) in the category of quasi-shuffle algebras is a bialgebra in the Hopfian category ${\QShi}$, where the Hopfian structure is induced by the $\star$ product. 
Concretely, it is a quasi-shuffle algebra $A$ equipped with a coassociative map $\Delta$ in ${\QShi}$ to $A\otimes A$, where the latter is equipped with a quasi-shuffle algebra structure by:
\begin{equation}\label{precbig}(a\otimes b)\prec (a'\otimes b')=(a\prec a')\otimes (b\star b'),
\end{equation}
\begin{equation}\label{dotbig}
(a\otimes b)\bullet (a'\otimes b')=(a\bullet a')\otimes (b\star b').\end{equation}
The same process defines the notion of shuffle bialgebra (without a unit), e.g. by taking a null $\bullet$ product in the definition.

Using Sweedler's shortcut notation $\Delta(a)=:a^{(1)}\otimes a^{(2)}$, one has:

\begin{equation}\label{qsbialgebra}
\Delta(a\prec b)=a^{(1)}\prec b^{(1)}\otimes a^{(2)}\star b^{(2)},\end{equation}
\begin{equation} \Delta(a\bullet b)=a^{(1)}\bullet b^{(1)}\otimes a^{(2)}\star b^{(2)}.
\end{equation}

In the unital case, $B=k\oplus A$, one requires furthermore that $\Delta$ be a counital coproduct (with $\Delta(1)=1\otimes 1$) and, since $1\prec 1$ and $1\bullet 1$ are not defined, sets:
$$(1\otimes b)\prec (1\otimes b')=1\otimes (b\prec b'),$$
$$(1\otimes b)\bullet (1\otimes b')=1\otimes (b\bullet b').$$
Since unital quasi-shuffle and shuffle bialgebras are more important for applications, we call them simply quasi-shuffle bialgebras and shuffle bialgebras. 
In this situation it is convenient to introduce the reduced coproduct on $A$,
$$\tdelta (a):=\Delta(a)-a\otimes 1-1\otimes a.$$ 
Concretely, we get:
\begin{defi}
The unital QSh algebra $k\oplus A$ equipped with counital coassociative coproduct $\Delta$ is a quasi-shuffle bialgebra if and only if
 for all $x,y\in A$ (we introduce for the reduced coproduct the Sweedler-type notation $\tdelta(x)=x'\otimes x''$):
\begin{equation}\label{qsbialgebra1}
\tdelta(x\prec y)=x'\prec y'\otimes x''\star y''+x'\otimes x''\star y+x\prec y'\otimes y''+x'\prec y\otimes x''+x\otimes y,
\end{equation}
\begin{equation}\label{qsbialgebra2}
\tdelta(x\bullet y)=x'\bullet y'\otimes x''\star y''+x'\bullet y\otimes x''+x\bullet y'\otimes y''.
\end{equation}
\end{defi}

The same constructions and arguments hold in the non-symmetric context. We do not repeat them and only state the conclusions.

\begin{prop}
 A non-symmetric category of binary algebras equipped with a forgetful functor to ${\Asi}$ is Hopfian. In particular, $\mathit{NSh}$ and ${\NQShi}$ are Hopfian.
\end{prop}

A  bialgebra (without a unit) in the category of noncommutative quasi-shuffle (NQSh) algebras is a bialgebra in the Hopfian category ${\NQShi}$, where the Hopfian structure is induced by the $\star$ product. 
Concretely, it is a NQSh algebra $A$ equipped with a coassociative map $\Delta$ in ${\NQShi}$ to $A\otimes A$, where the latter is equipped with a NQSh algebra structure by:
\begin{equation}\label{precbig2}(a\otimes b)\prec (a'\otimes b')=(a\prec a')\otimes (b\star b'),
\end{equation}
\begin{equation}\label{succbig2}(a\otimes b)\succ (a'\otimes b')=(a\succ a')\otimes (b\star b'),
\end{equation}
\begin{equation}\label{dotbig2}
(a\otimes b)\bullet (a'\otimes b')=(a\bullet a')\otimes (b\star b').\end{equation}
The same process defines the notion of NSh (or dendriform) bialgebra (without a unit), e.g. by taking a null $\bullet$ product in the definition. 

Recall that setting $\preceq :=\prec+\bullet$ defines a forgetful functor from NQSh to NSh algebras. The same definition yields a forgetful functor from NQSh to NSh bialgebras.

In the unital case, one requires furthermore that $\Delta$ be a counital coproduct (with $\Delta(1)=1\otimes 1$) and sets
$$(1\otimes b)\prec (1\otimes b')=1\otimes (b\prec b'),$$
and similarly for $\succ$ and $\bullet$. Since this case is more important for applications, we call simply NQSh and NSh bialgebras the ones with a unit. 

\begin{defi}
The unital NQSh algebra $k\oplus A$ equipped with counital coassociative coproduct $\Delta$ is a NQSh bialgebra if and only if
 for all $x,y\in A$:
\begin{equation}\label{nqsbialgebra1}
\tdelta(x\prec y)=x'\prec y'\otimes x''\star y''+x'\otimes x''\star y+x\prec y'\otimes y''+x'\prec y\otimes x''+x\otimes y,
\end{equation}
\begin{equation}\label{nqsbialgebra2}
\tdelta(x \succ y)=x'\succ y'\otimes x''\star y''+y'\otimes x\star y''+x\succ y'\otimes y''+x'\succ y\otimes x''+y\otimes x,
\end{equation}
\begin{equation}\label{nqsbialgebra3}
\tdelta(x\bullet y)=x'\bullet y'\otimes x''\star y''+x'\bullet y\otimes x''+x\bullet y'\otimes y''.
\end{equation}
\end{defi}

Recall, for later use, that a NQSh bialgebra $k\oplus A$ is \it connected \rm if the reduced coproduct is locally conilpotent:
$$A=\bigcup_{n\geq 0} Ker(\tdelta^{(n)}),$$
where $\tdelta^{(n)}$ is the iterated coproduct of order $n$ (and $Ker(\tdelta^{(n)})$ is also denoted $Prim(A)$, the set of primitive elements)
and similarly for the other unital bialgebras we will consider.

The reason for the importance of the unital case comes from Hoffman's:
\begin{exam} Let $V$ be an associative, non unitary algebra. With the deconcatenation coproduct $\Delta$, defined by:
$$\Delta(x_1\ldots x_n)=\sum_{i=0}^{n} x_1\ldots x_i \otimes x_{i+1}\ldots x_n,$$
 the tensor quasi-shuffle algebra $T(V)$ is a NQSh bialgebra. When $V$ is commutative, it is a quasi-shuffle bialgebra.
\end{exam}

\section{Lie theory for quasi-shuffle bialgebras}\label{sect:4}
The structural part of Lie theory, as developed for example in Bourbaki's \it Groupes et Alg\`ebres de Lie \rm \cite{Bourbaki} and Reutenauer's monograph on free Lie algebras \cite{Reutenauer}, is largely concerned with the structure of enveloping algebras and cocommutative Hopf algebras. It was shown in \cite{Livernet} that many phenomena that might seem characteristic of Lie theory do actually generalize to other families of bialgebras -precisely the ones studied in the previous section, that is the ones associated with Hopfian categories of algebras equiped with a forgetful functor to ${\Comi}$ or ${\Asi}$.

The most natural way to study these questions is by working with twisted algebras over operads --algebras in the category of $\mathbf S$-modules (families of representations of all the symmetric groups $\S_n,\ n\geq 0$) or, equivalently, of functors from finite sets to vector spaces. However, doing so systematically requires the introduction of many terms and preliminary definitions (see \cite{Livernet}), and we prefer to follow here a more direct approach inspired by the theory of combinatorial Hopf algebras. The structures we are going to introduce are reminiscent of the Malvenuto--Reutenauer Hopf algebra \cite{MR}, whose construction can be deduced from the Hopfian structure of ${\Asi}$, see \cite{PR04,patras2006twisted,patras2008trees} and \cite[Exple 2.3.4]{Livernet}. The same process will allow us to contruct a combinatorial Hopf algebra structure on the operad ${\QSh}$ of quasi-shuffle algebras.

Recall that an algebraic theory such as the ones we have been studying (associative, commutative, quasi-shuffle, NQSh... algebras) is entirely characterized by the behaviour of the corresponding free algebra functor $F$: an analytic functor described by a sequence of symmetric group representation $\mathbf{F}_n$ (i.e. a $\mathbf S$-module)
so that, for a vector space $V$, $F(V)=\bigoplus\limits_n\mathbf{F}_n\otimes_{\S_n}V^{\otimes n}.$ Composition of operations for $F$-algebras are encoded by natural transformations from $F\circ F$ to $F$. By a standard process, this defines a monad, and $ F$-algebras are the algebras over this monad.
The direct sum $\mathbf{F}=\bigoplus\limits_n\mathbf{F}_n$ equipped with the previous (multilinear) composition law is called an operad, and $F$-algebras are algebras over this operad.
Conversely, the $\mathbf{F}_n$ are most easily described as the multilinear part of the free $F$-algebras $F(X_n)$ over the vector space spanned by a finite set with $n$ elements, $X_n:=\{x_1,\dots,x_n\}$. Here, multilinear means that ${\mathbf F}_n$ is the intersection of the $n$ eigenspaces associated to the eigenvalue $\lambda$ of the $n$ operations induced on $F(X_n)$ by the map that scales $x_i$ by $\lambda$ (and acts as the identity on the $x_j,\ j\not= i$).

Let $X$ be a finite set, and let us anticipate on the next Lemma and write $QSh(X):=T^+(k[X]^+)$ for the quasi-shuffle algebra associated to $k[X]^+$, the (non unital, commutative) algebra of polynomials without constant term over $X$. For $I$ a multiset over $X$, we write $x_I$ the associated monomial 
(e.g. if $I=\{x_1,x_3,x_3\}$, $x_I=x_1x_3^2$). The tensors $x_{I_1}\dots x_{I_n}=x_{I_1}\otimes \dots \otimes x_{I_n}$ form a basis of $QSh(X)$.

There are several ways to show that $QSh(X)$ is the free quasi-shuffle algebra over $X$: the property can be deduced from the classical constructions of commutative Rota-Baxter algebras by Cartier \cite{Cartier} or Rota \cite{Rota1,Rota2} (indeed the tensor product $x_{I_1}\dots x_{I_n}$ corresponds to the Rota--Baxter monomial $x_{I_1}R(x_{I_2}R(x_{I_3}\dots R( x_{I_n})\dots )))$ in the free RB algebra over $X$). It can be deduced from the construction of the free shuffle algebra over $X$ by standard filtration/graduation arguments. It can also be deduced from a Schur functor argument \cite{loday3}.
The simplest proof is but the one due to Sch\"utzenberger for shuffle algebras that applies almost without change to quasi-shuffle algebras \cite[p. 1-19]{schutz}. 

\begin{lemma}
The quasi-shuffle algebra $QSh(X)$ is the (unique up to isomorphism) free quasi-shuffle algebra over $X$.
\end{lemma}

Indeed, let $A$ be an arbitrary quasi-shuffle algebra generated by $X$. Then, one checks easily by a recursion using the defining relations of quasi-shuffle algebras that every $a\in A$ is a finite sum of ``normed terms'', that is terms of the form
$$x_{I_1}\prec (x_{I_2}\prec (x_{I_3}\dots \prec x_{I_n})\dots ).$$
But, if $A=QSh(X)$, by the Sch\"utzenberger's trick,
$x_{I_1}\prec (x_{I_2}\prec (x_{I_3}\dots \prec x_{I_n})\dots )=x_{I_1}\dots x_{I_n}$; the result follows from the fact that these terms form a basis of $QSh(X)$.

\begin{cor}The component ${\QSh}_n$  of the operad ${\QSh}$ identifies therefore with the linear span of tensors $x_{I_1}\dots x_{I_k}$, where $I_1\coprod \dots \coprod I_k=[n]$. 
\end{cor}

Let us introduce useful notations. 
We write $x_{\mathcal I}:=x_{I_1}\dots x_{I_k}$, where $\mathcal I$ denotes an arbitrary ordered sequence of disjoint subsets of $\mathbf{N}^\ast$, $I_1, \dots , I_k$, and set $|\mathcal I|:=|I_1|+\dots +|I_k|$. 
Recall that the standardization map associated to a subset $I=\{i_1,\dots,i_n\}$ of $\mathbf N^\ast$, where $i_1<\dots <i_n$ is the map $st$ from $I$ to $[n]$ defined by: $st(i_k):=k.$ The \it standardization \rm of $\mathcal I$ is then the ordered sequence $st(\mathcal I):= st(I_1,\dots ,I_k)$, where $st$ is the standardization map associated to the subset 
$I_1\coprod\dots \coprod I_k$ of the integers.
We also set $st(x_{\mathcal I}):=x_{st(\mathcal I)}$. For example, if $\mathcal I=\{2,6\},\{5,9\}$, $st({\mathcal I})=\{1,3\},\{2,4\}$ and
$st(x_{\mathcal I})=x_1x_3\otimes x_2x_4$.
The shift by $k$ of a subset $I=\{i_1,\dots,i_n\}$ (or a sequence of subsets, and so on...) of $\mathbf N^\ast$, written $I+k$, is defined by $I+k:=\{i_1+k,\dots,i_n+k\}$.

\begin{theo}\label{Qshbialg}
The operad ${\QSh}$ of quasi-shuffle algebras inherits from the Hopfian structure of its category of algebras a NQSh bialgebra structure whose product operations are defined by:
$$x_{\mathcal I}\prec x_{\mathcal J}:=x_{\mathcal I}\prec_f x_{\mathcal J+n},$$
$$x_{\mathcal I}\succ x_{\mathcal J}:=x_{\mathcal I}\succ_f x_{\mathcal J+n},$$
$$x_{\mathcal I}\bullet x_{\mathcal J}:=x_{\mathcal I}\bullet_f x_{\mathcal J+n},$$
where $\mathcal I$ and $\mathcal J$ run over ordered partitions of $[n]$ and $[m]$;
the coproduct is defined by:
$$\Delta(x):=(st\otimes st)\circ\Delta_f (x),$$
where, on the right-hand sides, $\prec_f,\succ_f,\bullet_f, \Delta_f$ stand for the corresponding operations on $QSh(\mathbf N^\ast)$ (where, as usual, $x\prec_f y=:y\succ_f x$).
\end{theo}

The link with the Hopfian structure of the category of quasi-shuffle algebras refers to \cite[Thm 2.3.3]{Livernet}: any connected Hopf operad is a twisted Hopf algebra over this operad. The Theorem \ref{Qshbialg} can be thought of as a reformulation of this general result in terms of NQSh bialgebras.

The fact that ${\QSh}$ is a NQSh algebra follows immediately from the fact that $QSh(\mathbf N^\ast)$ is a NQSh algebra for $\prec_f,\succ_f,\bullet_f$, together with the fact that the category of NQSh algebras is non-symmetric.
The coalgebraic properties and their compatibility with the NQSh algebra structure are less obvious and follow from the following Lemma (itself a direct consequence of the definitions):

\begin{lemma}
Let $\mathcal I =I_1,\dots ,I_k$ and $\mathcal J=J_1,\dots ,J_l$ be two ordered sequence of disjoint subsets of $\mathbf{N}^\ast$  that for any $n\in \mathcal I_p,\ p\leq k$ and any $m\in \mathcal J_q,\ q\leq l$ we have $n<m$. Then:
$$st(x_{\mathcal I}\prec_f x_{\mathcal J})=x_ {st({\mathcal I})}\prec_f x_{st(\mathcal J)+|\mathcal I|}=x_ {st({\mathcal I})}\prec x_{st(\mathcal J)},$$
$$st(x_{\mathcal I}\succ_f x_{\mathcal J})=x_ {st({\mathcal I})}\succ_f x_{st(\mathcal J)+|\mathcal I|}=x_ {st({\mathcal I})}\succ x_{st(\mathcal J)},$$
$$st(x_{\mathcal I}\bullet_f x_{\mathcal J})=x_ {st({\mathcal I})}\bullet_f x_{st(\mathcal J)+|\mathcal I|}=x_ {st({\mathcal I})}\bullet x_{st(\mathcal J)}.$$
\end{lemma}

The Hopf algebra ${\QSh}$ is naturally isomorphic with $\mathbf{WQSym}$, the Chapoton-Hivert Hopf algebra of word quasi-symmetric functions, that has been studied in \cite{novelli2,foissy2}, also in relation to quasi-shuffle algebras, but from a different point of view.

\ \par

Let us conclude this section by some insights on the "Lie theoretic" structure underlying the previous constructions on ${\QSh}$ (where "Lie theoretic" refers concretely to the behaviour of the functor of primitive elements in a class of bialgebras associated to an Hopfian category with a forgetful functor to ${\Asi}$ or ${\Comi}$).
Recall that there is a forgetful functor from quasi-shuffle algebras to commutative algebras defined by keeping only the $\bullet$ product. Dualy, the operad $\Com$ embeds into the operad ${\QSh}$: $\Com_n$ is the vector space of dimension 1 generated by the monomial $x_1\dots x_n$, and through the embedding into ${\QSh}$ this monomial is sent to the monomial (a tensor of length 1) $x_1^{\bullet n}:=x_1\bullet\dots\bullet x_1$ in $\QSh$ viewed as a NQSh algebra. Let us write slightly abusively $\Com$ for the image of $\Com$ in ${\QSh}$, we have, by definition of the coproduct on ${\QSh}$:

\begin{theo}
The operad $\Com$ embeds into the primitive part of the operad ${\QSh}$ viewed as a NQSh bialgebra. Moreover, the primitive part of ${\QSh}$ is stable under the $\bullet$ product.
\end{theo}

Only the last sentence needs to be proved. It follows from the relations:
$$1\bullet x=x\bullet 1=0$$
for $x\in {\QSh}_n,\ n\geq 1$.

From the point of view of $\mathbf S$-modules, the Theorem should be understood in the light of \cite[Thm 2.4.2]{Livernet}: for $\mathbf P$ a connected Hopf operad, the space of primitive elements of the twisted Hopf $ P$-algebra $\mathbf P$ is a sub-operad of $\mathbf P$.

\ \par 

As usual in categories of algebras a forgetful functor such as the one from ${\QShi}$ to ${\Comi}$ induced by $\bullet$ has a left adjoint, see e.g. \cite{GJ94} for the general case and \cite{loday3} for quasi-shuffle algebras. 
This left adjoint, written $U$ (by analogy with the case of classical enveloping algebras: $U(A)\in {\QShi}$ for $A\in {\Comi}$ equipped with a product written $\cdot$) is, up to a canonical isomorphism, the quotient of the free quasi-shuffle over the vector space $A$ by the relations $a\bullet b=a\cdot b$.
When the initial category is Hopfian, such a forgetful functor to a category of algebras over a naturally defined sub-operad arises from the properties of the tensor product of algebras in the initial category, see \cite[Thm 2.4.2 and Sect. 3.1.2]{Livernet} --this is exactly what happens with the pair $({\Asi},{\mathit{Lie}})$ in the classical situation where the left adjoint is the usual enveloping algebra functor, and here for the pair $({\QShi},{\Comi})$.

\begin{lemma}[Quasi-shuffle PBW theorem]
The left adjoint $U$ of the forgetful functor from ${\QShi}$ to ${\Comi}$, or "quasi-shuffle enveloping algebra" functor from ${\Comi}$ to ${\QShi}$, is (up to isomorphism) Hoffman's quasi-shuffle algebra functor $T^+$.
\end{lemma}

An elementary proof follows once again from (a variant of) Sch\"utzenberger's construction of the free shuffle algebra.
Notice first that $T^+(A)$ is generated by $A$ as a quasi-shuffle algebra, and that, in it, the relations $a\bullet b=a\cdot b$ hold. Moreover, choosing a basis $(a_i)_{i\in I}$ of $A$, the tensors $a_{i_1}\dots a_{i_n}=a_{i_1}\prec (a_{i_2}\prec \dots \prec a_{i_n})\dots )$ form a basis of $T^+(A)$. On the other hand,
by the definition of the left adjoint $U(A)$ as a quotient of $Sh(A)$ by the relations $a\bullet b=a\cdot b$, using the defining relations of quasi-shuffle algebras, any term in $U(A)$ can be written recursively as a sum of terms in "normed form" $a_{i_1} \prec (a_{i_2}\prec \dots (a_{i_{n-1 }}\prec a_{i_n})\dots )$. The Lemma follows.

Notice that the existence of a basis of $T^+(A)$ of tensors $a_{i_1}\dots a_{i_n}=a_{i_1}\prec (a_{i_2}\prec \dots \prec a_{i_n})\dots )$ is the analog, for quasi-shuffle enveloping algebras, of the Poincar\'e-Birkhoff-Witt (PBW) basis for usual enveloping algebras.

\section{Endomorphism algebras}\label{sect:5}

We follow once again the analogy with the familiar notion of usual enveloping algebras and connected cocommutative Hopf algebras and study, in this section the analogs of the convolution product of their linear endomorphisms. Surjections happen to play, for quasi-shuffle algebras $T(A)$ associated to commutative algebras $A$, the role played by bijections in classical Lie theory, see  \cite{MR} and \cite{novelli2,foissy2}.

\begin{prop}Let $A$ be a coassociative (non necessarily counitary) coalgebra with coproduct $\tdelta:A\longrightarrow A\otimes A$, and $B$ be a NQSh algebra. The space of linear morphisms $Lin(A,B)$ is given a NQSh algebra structure in the following way: for all $f,g\in Lin(A,B)$, 
\begin{align}
f\prec g&=\prec \circ (f\otimes g) \circ \tdelta,&
f\succ g&=\succ \circ (f\otimes g) \circ \tdelta,&
f\bullet g&=\bullet \circ (f\otimes g) \circ \tdelta.
\end{align}\end{prop}

\begin{proof} The construction follows easily from the fact that ${\NQShi}$ is non-symmetric and from the coassociativity of the coproduct.
As an example, let us prove (\ref{E2}) using Sweedler's notation for $\tdelta$. Let $f,g,h\in Lin(A,B)$. For all $x\in A$,
\begin{align*}
(f\succ g)\prec h(x)&=(f\succ g)(x') \prec h(x'')\\
&=(f((x')')\succ g((x')''))\prec h(x'')\\
&=f(x') \succ (g((x'')')\prec h((x'')'')\\
&=f(x') \succ (g\prec h)(x'')\\
&=f\succ (g\prec h)(x).
\end{align*}
So $(f\succ g)\prec h=f\succ (g\prec h)$. \end{proof} 

{\bf Remark.} The induced product $\star$ on $Lin(A,B)$ is the usual convolution product.

\begin{cor}
The set of linear endomorphisms of $A$, where $k\oplus A$ is a NQSh bialgebra is naturally equiped with the structure of a NQSh algebra.
\end{cor}

Let us turn now to the quasi-shuffle analog of the Malvenuto-Reutenauer noncommutative shuffle algebra of permutations. The appearance of a noncommutative shuffle algebra of permutations in Lie theory in \cite{MR} can be understood operadically by noticing that the linear span of the $n$-th symmetric group $\S_n$ is $\As_n$, the $n$-th component of the operad of associative algebras (or also the $n$-th component of $\QSh$, the quasi-shuffle operad). The same reason explain why surjections appear naturally in the study of quasi-shuffle algebras: ordered partitions of initial subsets of the integers (say $\{2,4\},\{5\},\{1,3\}$) parametrize  a natural basis of ${\QSh}_n$, and such ordered partitions are canonically in bijection with surjections (here, the surjection $s$ from $[5]$ to $[3]$ defined by $s(2)=s(4)=1,\ s(5)=2,\ s(1)=s(3)=3$). Let us show how the NQSh algebra structure of ${\QSh}$ can be recovered from the point of view of the structure of NQSh algebras of linear endomorphisms. In the process, we also give explicit combinatorial formulas for the corresponding structure maps $\prec,\succ,\bullet$. We also point out that composition of endomorphisms leads to a new product on ${\QSh}$ (such a product is usually called ``internal product'' in the theory of combinatorial Hopf algebras, we follow the use, see \cite{gelfand1994noncommutative,novelli2}).

Let $n \geq 0$. We denote by $\Surj_n$ the set of maps $\sigma:[n]:=\{1,\ldots,n\}\longrightarrow \mathbb{N}^\ast$
such that $\sigma(\{1,\ldots,n\})=\{1,\ldots,k\}$ for a certain $k$. The corresponding elements in ${\QSh}_n$ are the ordered partitions $\sigma^{-1}(\{1\}),\dots,\sigma^{-1}(\{k\})$ of $[n]$. The integer $k$ is the maximum of $\sigma$ and denoted by $max(\sigma)$.
 The element $\sigma\in \Surj_n$ will be represented by the packed word $(\sigma(1)\ldots \sigma(n))$ (recall that a word $n_1\dots n_k$ over the integers is called packed if the underlying set $S=\{n_1,\dots,n_k\}$ is an initial subset of $\mathbb{N}^\ast$, that is, $S=[m]$ for a certain $m$).  
 We identify in this way elements of $\Surj_n$ with packed words of length $n$. For later use, recall also that any word $n_1\dots n_k$ over the integers can be packed: $pack(n_1\dots n_k)=m_1\dots m_k$ is the unique packed word preserving the natural order of letters ($m_i<m_j\Leftrightarrow n_i<n_j$, $m_i=m_j\Leftrightarrow n_i=n_j$, e.g. $pack(6353)=3121$). \\

We assume that $V$ is an associative, commutative algebra and work with the quasi-shuffle algebra $T^+(V)$.
Let $\sigma \in \Surj_n$, $n \geq 1$.  We define $F_{\sigma} \in End_K(T(V))$ in the following way: for all $x_1,\ldots,x_l \in V$,
$$F_{\sigma}(x_1\ldots x_l)=\left\{
\begin{array}{l}
\displaystyle \left(\prod_{\sigma(i)=1}x_i\right)\ldots \left(\prod_{\sigma(i)=max(\sigma)}x_i\right) \mbox{ if }k=l, \\[9mm]
0\mbox{ otherwise}.
\end{array}\right.$$
Note that in each parenthesis, the product is the product of $V$. For example, if $x,y,z\in V$,
\begin{align*}
F_{(123)}(xyz)&=xyz&F_{(132)}(xyz)&=xzy&F_{(213)}(xyz)&=yxz\\
F_{(231)}(xyz)&=zxy&F_{(312)}(xyz)&=yzx&F_{(321)}(xyz)&=zyx\\
F_{(122)}(xyz)&=x(y.z)&F_{(212)}(xyz)&=y(x.z)&F_{(221)}(xyz)&=z(x.y)\\
F_{(112)}(xyz)&=(x.y)z&F_{(121)}(xyz)&=(x.z)y&F_{(211)}(xyz)&=(y.z)x\\
&&F_{(111)}(xyz)&=x.y.z.
\end{align*}
We also define $F_1$, where $1$ is the empty word, by $F_1(x_1\ldots x_n)=\varepsilon(x_1\ldots x_n)1=\delta_{x_1\ldots x_n,1}$.\\

{\bf Notations.} Let $k,l \geq 0$.
\begin{enumerate}
\item \begin{enumerate}
\item  We denote by $QSh_{k,l}$ the set of $(k,l)$ quasi-shuffles, that is to say elements $\sigma\in \Surj_{k+l}$ such that
$\sigma(1)<\ldots <\sigma(k)$ and $\sigma(k+1)<\ldots<\sigma(k+l)$.
\item $QSh^\prec_{k,l}$ is the set of $(k,l)$ quasi-shuffles $\sigma$ such that $\sigma^{-1}(\{1\})=\{1\}$.
\item $QSh^\succ_{k,l}$ is the set of $(k,l)$ quasi-shuffles $\sigma$ such that $\sigma^{-1}(\{1\})=\{k+1\}$.
\item $QSh^\bullet_{k,l}$ is the set of $(k,l)$ quasi-shuffles $\sigma$ such that $\sigma^{-1}(\{1\})=\{1,k+1\}$.
\end{enumerate}
Note that $QSh_{k,l}=QSh^\prec_{k,l}\sqcup QSh^\succ_{k,l}\sqcup QSh^\bullet_{k,l}$.
\item If $\sigma \in \Surj_k$ and $\tau \in \Surj_l$, $\sigma \otimes \tau$ is the element of $\Surj_ {k+l}$ represented by
the packed word $\sigma \tau[\max(\sigma)]$., where $[k]$ denotes the translation by $k$ ($312[5]=867$).
\end{enumerate}

The subspace of $End_K(T(V))$ generated by the maps $F_\sigma$ is stable under composition and the  products:

\begin{prop}\label{7}
Let $\sigma \in \Surj_k$ and $\tau \in \Surj_l$.
\begin{enumerate}
\item If $max(\tau)=k$, then $F_\sigma\circ F_\tau=F_{\sigma \circ \tau}$. Otherwise, this composition is equal to $0$.
\item  \begin{align*}
F_\sigma\prec F_\tau&=\sum_{\zeta \in QSh_{k,l}^\prec}F_{\zeta \circ (\sigma \otimes \tau)},&
F_\sigma\succ F_\tau&=\sum_{\zeta \in QSh_{k,l}^\succ}F_{\zeta \circ (\sigma \otimes \tau)},\\
F_\sigma\bullet F_\tau&=\sum_{\zeta \in QSh_{k,l}^\bullet}F_{\zeta \circ (\sigma \otimes \tau)},&
F_\sigma\qshuffle F_\tau&=\sum_{\zeta \in QSh_{k,l}}F_{\zeta \circ (\sigma \otimes \tau)}.
\end{align*}
The same formulas describe the structure of the operad ${\QSh}$ as a NQSh algebra (i.e., in $\QSh$, using the identification between surjections and ordered partitions, $\sigma\prec \tau =\sum_{\zeta \in QSh_{k,l}^\prec}\zeta \circ (\sigma \otimes \tau)$, and so on).
\end{enumerate}\end{prop}

\begin{proof} The proof of 1. and 2. follows by direct computations. The identification with the corresponding formulas for $\QSh$ follows from the identities, 
for all $x_1,\ldots,x_{k+l} \in V$, in the quasi-shuffle algebra $T^+(A)$:
\begin{align*}
x_1\ldots x_k \prec x_{k+1}\ldots x_{k+l}&=\sum_{\zeta \in QSh_{k,l}^\prec}F_\zeta(x_1\ldots x_{k+l}),\\
x_1\ldots x_k \succ x_{k+1}\ldots x_{k+l}&=\sum_{\zeta \in QSh_{k,l}^\succ}F_\zeta(x_1\ldots x_{k+l}),\\
x_1\ldots x_k \bullet x_{k+1}\ldots x_{k+l}&=\sum_{\zeta \in QSh_{k,l}^\bullet}F_\zeta(x_1\ldots x_{k+l}),\\
x_1\ldots x_k \qshuffle x_{k+1}\ldots x_{k+l}&=\sum_{\zeta \in QSh_{k,l}}F_\zeta(x_1\ldots x_{k+l}).
\end{align*}
Moreover:
$$x_1\ldots x_k \shuffle x_{k+1}\ldots x_{k+l}=\sum_{\zeta \in Sh_{k,l}}F_\zeta(x_1\ldots x_{k+l}),$$
where $Sh_{k,l}$ is the set of $(k,l)$-shuffles, that is to say $\S_{k+l}\cap QSh_{k,l}$.

 \end{proof}

{\bf Remarks.} \begin{enumerate}
\item $F_{(1\ldots n)}$ is the projection on the space of words of length $n$. Consequently:
$$Id=\sum_{n=0}^\infty F_{(1\ldots n)}.$$
\item In general, this action of  packed words is not faithful. For example, if $A$ is a trivial algebra, 
then for any $\sigma \in \Surj_k\setminus {\S}_k$, $F_\sigma=0$.
\item Here is an example where the action is faithful. Let $A=K[X_i\mid i\geq 1]_+$.
Let us assume that $\sum a_\sigma F_\sigma=0$. Acting on the word $X_1\ldots X_k$, we obtain:
$$\sum_{\sigma \in \Surj_k}a_\sigma\left(\prod_{\sigma(i)=1} X_i\right)\ldots \left(\prod_{\sigma(i)=max(\sigma)}X_i\right)=0.$$
As the $X_i$ are algebraically independent, the words appearing in this sum are linearly independent, so for all $\sigma$, $a_\sigma=0$. 
\end{enumerate}

\section{Canonical projections on primitives}\label{sect:6}

This section studies the analog, for quasi-shuffle bialgebras, of the canonical projection from a connected cocommutative Hopf algebra to its primitive part --the logarithm of the identity (see e.g. \cite{Reutenauer,patras1993decomposition,patras1994algebre}).

Recall that a coalgebra with a coassociative coproduct $\tilde\Delta$ is \it connected \rm if and only if the coproduct il locally conilpotent.

\begin{prop}\label{4}Let $A$ be a coassociative, non counitary, coalgebra with a locally conilpotent coproduct $$\tdelta:A\longrightarrow A\otimes A,\ \  A=\bigcup_{n\geq 0} Ker(\tdelta^{(n)}),$$ and let $B$ be a NQSh algebra. Then, for any $f\in Lin(A,B)$,
 there exists a unique map $\pi_f\in Lin(A,B)$,
such that $$f=\pi_f+\pi_f \prec f.$$
\end{prop}

\begin{proof} For all $n\geq 1$, we put $F_n=Ker(\tdelta^{(n)})$: this defines the coradical filtration of $A$.
In particular, $F_1=:Prim(A)$.  Moreover, if $n\geq 1$:
$$\tdelta(F_n)\subseteq F_{n-1}\otimes F_{n-1}.$$
Let us choose for all $n$ a subspace $E_n$ of $A$ such that $F_n=F_{n-1}\oplus E_n$. In particular, $E_1=F_1=Prim(A)$.
Then, $A$ is the direct sum of the $E_n$'s and for all $n$:
$$\tdelta(E_n)\subseteq \bigoplus_{i,j<n} E_i \otimes E_j.$$

{\it Existence.} We inductively define a map $\pi_f:E_n \longrightarrow B$ for all $n \geq 1$ in the following way:
\begin{itemize}
\item For all $a\in E_1$, $\pi_f(a)=f(a)$.
\item If $a\in E_n$, as $\displaystyle \tdelta(a)\in \bigoplus_{i+j<n} E_i \otimes E_j$, $(\pi_f \otimes f)\circ \tdelta(a)$ is already defined. We then put:
$$\pi_f(a)=f(a)-\prec \circ (\pi_f \otimes f) \circ \tdelta(a) = f(a)-(\pi_f \prec f)(a)$$
\end{itemize}

{\it Unicity.} Let $\mu_f$ such that $f=\mu_f+(\mu_f \prec f)$. For all $a\in E_1$,
$f(a)=\mu_f(a)+0$, so $\mu_f(a)=\pi_f(a)$. Let us assume that for all $k<n$, $\mu_f(a)=\pi_f(a)$ if $a\in E_k$. Let $a\in E_n$. Then:
$$a=\mu_f(a)+\mu_f(a')\prec a''=\mu_f(a)+\pi_f(a')\prec a''=\mu_f(a)+a-\pi_f(a),$$
so $\mu_f(a)=\pi_f(a)$. Hence, $\mu_f=\pi_f$. 
\end{proof}

\begin{prop}\label{8}
When $A=B=T^+(V)$ and $f=Id$, the map $\pi:=\pi_f$ defined in proposition \ref{4}  is equal to the projection $F_{(1)}$.
\end{prop}

\begin{proof} First, observe that, as $QSh^\prec_{1,k}=\{(1,\ldots,k)\}$, 
for all packed words $(a_1\ldots a_k)$, $F_{(1)}\prec F_{(a_1\ldots a_k)}=F_{(1 a_1+1\ldots a_k+1)}$.
Hence, in $A$:
 $$ F_{(1)}+F_{(1)}\prec Id_A=F_{(1)}+\sum_{n=1}^\infty F_{(1)}\prec F_{(1\ldots n)}=F_{(1)}
+\sum_{n=1}^\infty F_{(1\ldots n+1)}=\sum_{n=1}^\infty F_{(1\ldots n)}=Id_A.$$
 By unicity in proposition \ref{4}, $\pi_f=F_{(1)}$. \end{proof}

More generally, we have:
\begin{prop}\label{4'}
Let $A$ be a non unital, connected NQSh bialgebra, and $\pi$ the unique solution to
$$Id_A=\pi+\pi\prec Id_A,$$
then $\pi$ is a projection on $Prim(A)$, and for all $x\in Prim(A)$,
$y\in A$, $\pi(x\prec y)=0$.
\end{prop}

\begin{proof} Let us prove that for all $a\in E_n$, $\pi(a)\in Prim(A)$ by induction on $n$.
As $E_1=Prim(A)$, this is obvious if $n=1$. Let us assume the result for all $k<n$. Let $a \in E_n$. Then $\pi(a)=a-\pi(a')\prec a''$.
By the induction hypothesis, we can assume that $\pi(a') \in Prim(A)$, so:
$$\tdelta(\pi(a))=a'\otimes a''-\pi(a')\prec a''\otimes a'''-\pi(a')\otimes a''=(a'-(\pi \prec Id)(a')-\pi(a'))\otimes a''=0.$$
Hence, for all $a\in A$, $\pi(a)\in Prim(a)$. So $\pi$ that, by its very  definition, acts as the identity on $Prim(A)$, is a projection on $Prim(A)$. 

Let $x \in Prim(A)$ and $y\in E_n$, let us prove that $\pi(x\prec y)=0$ by induction on $n$. If $n=1$, then $y\in Prim(A)$, so $\tdelta(x\prec y)=x\otimes y$,
and $\pi(x\prec y)=x\prec y-\pi(x)\prec y=x\prec y-x\prec y=0$. Let us assume the result at all rank $<n$. We have:
$$\tdelta(x\prec y)=x\prec y'\otimes y''+x\otimes y.$$
By the induction hypothesis, we can assume that $\pi(x\prec y')=0$, so $\pi(x\prec y)=x\prec y-0-\pi(x)\prec y=x\prec y-x\prec y=0$. \end{proof} 

{\bf Remark.} For all $x,y\in Prim(A)$:
\begin{align*}
\pi(x\prec y)&=0,&
\pi(x\succ y)&=x\succ y-y\prec x,&
\pi(x\bullet y)&=x\bullet y.
\end{align*}

\begin{cor}
Let $A$ be a nonunital, connected quasi-shuffle bialgebra. Then $Prim(A)$ is stable under $\bullet$ and the following map is an isomorphism
of quasi-shuffle bialgebras:
$$\theta:\left\{\begin{array}{rcl}
T^+(Prim(A))&\longrightarrow&A\\
a_1\ldots a_k&\longrightarrow&a_1\prec (a_2\prec (\ldots \prec a_k)\ldots).
\end{array}\right.$$
\end{cor}

\begin{proof} 
Let  $a_1,\ldots,a_k \in Prim(A)$. An easy induction on $k$ proves that:
$$\tdelta(\theta(a_1\otimes \ldots \otimes a_k))=\sum_{i=1}^{k-1} \theta(a_1\otimes \ldots \otimes a_i) \otimes \theta(a_{i+1}\otimes \ldots
\otimes a_k).$$
So $\theta$ is a coalgebra morphism.

From this coalgebra morphism property and the identity $\pi (x\prec y)=0$ for $x\in Prim(A)$, we get for $a_1,\ldots,a_k \in Prim(A)$, $(Id_A\otimes \pi)\circ \tdelta(\theta(a_1\otimes \ldots \otimes a_k))=\theta(a_1\otimes \ldots \otimes a_{k-1}) \otimes \theta(a_k)$. Since $\theta$ is the identity on its restriction to $Prim(A)$, its injectivity follows by induction.\\

Let $a=a_1\ldots a_k$ and $b=b_1\ldots b_l \in T^+(Prim(A))$. Let us prove by induction on $k+l$ that:
\begin{align*}
\theta(a\prec b)&=\theta(a)\prec \theta(b),&\theta(a\succ b)&=\theta(a)\succ \theta(b),&\theta(a\bullet b)&=\theta(a)\bullet \theta(b).
\end{align*}
If $k=1$, then $a\prec b_1\ldots b_l=ab_1\ldots b_l$, so $\theta(a\prec b)=a\prec \theta(b)=\theta(a)\prec \theta(b)$.
If $l=1$, then $a\succ b=ba$, so $\theta(a\succ b)=b \prec \theta(a)=\theta(b) \prec \theta(a)=\theta(a)\succ \theta(b)$.
If $k=l=1$, , $x\bullet y=\pi (x\bullet y) \in Prim(A)$, so $\theta(a \bullet b)=a\bullet b=\theta(a)\bullet \theta(b)$.
All these remarks give the results for $k+l\leq 2$. Let us assume the result at all ranks $<k+l$. 
If $k=1$, we already proved that $\theta(a\prec b)=\theta(a)\prec \theta(b)$. If $k\geq 2$, $a\prec b=a_1 (a_2\ldots a_k \qshuffle b)$.
By the induction hypothesis applied to $a_2\ldots a_k$ and $b$:
$$\theta(a\prec b)=a_1\prec (\theta(a_2\ldots a_k) \star \theta(b))=(a_1\prec \theta(a_2\ldots a_k)) \prec \theta(b)
=\theta(a)\prec \theta(b).$$
Using the commutativity of $T^+(Prim(A))$ and $A$, we obtain $\theta(a\succ b)=\theta(a)\succ \theta(b)$. 
If $l>1$, $a\bullet b=a\bullet (b_1\prec b_2\ldots b_l)=(a\bullet b_1)\prec b_2\ldots b_l$.
Moreover, $a\bullet b_1$ is a linear span of words of length $\leq k+1$, so, by the preceding computation and the induction hypothesis:
$$\theta(a\bullet b)=\theta(a\bullet b_1)\prec \theta(b_2\ldots b_l).$$
The induction hypothesis holds for $a$ and $b_1$, so:
$$\theta(a\bullet b)=(\theta(a)\bullet \theta(b_1))\prec \bullet(b_2\ldots b_l)
=\theta(a) \bullet (b_1 \prec \theta(b_2\ldots b_l))=\theta(a)\bullet \theta(b).$$
If $l=1$, then $k>1$ and we conclude with the commutativity of $\bullet$. \\

Let us now prove that $Prim(A)$ generates $A$ as a quasi-shuffle algebra.
Let $A'$ be the quasi-shuffle subalgebra of $A$ generated by $Prim(A)$. Let $a\in E_n$, let us prove that $x \in A'$ by induction on $n$.
As $E_1=Prim(A)$, this is obvious if $n=1$. Let us assume the result for all ranks $<n$. Then $a=\pi(a)+\pi(a')\prec a''$.
By the induction hypothesis, $a''\in A'$. Moreover, $\pi(a)$ and $\pi(a')\in Prim(A)$, so $a\in A'$.  

As a conclusion, $\theta$ is a morphism of quasi-shuffle algebras, whose image contains $Prim(A)$, which generates $A$, so $\theta$ is surjective.
 \end{proof}

\section{Relating the shuffle and quasi-shuffle operads}\label{sect:7}

A fundamental theorem of the theory of quasi-shuffle algebras relates quasi-shuffle bialgebras and shuffle bialgebras and, under some hypothesis (combinatorial and graduation hypothesis on the generators in Hoffman's original version of the theorem \cite{Hoffman}), shows that the two categories of bialgebras are isomorphic. This result allows to understand quasi-shuffle bialgebras as deformations of shuffle bialgebras and, as such, can be extended to other deformations of the shuffle product than the one induced by Hoffman's exponential map, see \cite{foissy2}. We will come back to this line of arguments in the next section.

Here, we stick to the relations between shuffle and quasi-shuffle algebras and show that Hoffman's theorem can be better understood and refined in the light of an Hopf algebra morphism relating the shuffle and quasi-shuffle operads.

Let us notice first that the same construction that allows to define a NQSh algebra structure on the operad ${\QSh}$ allows, \it mutatis mutandis\rm , to define a noncommutative shuffle algebra structure on $\Sh$, the operad of shuffle algebras. A natural basis of  the latter operad is given by permutations (the result goes back to Sch\"utzenberger, who showed that the tensor algebra over a vector space $V$ is a model of the free shuffle algebra over $V$ \cite{schutz}). Let us stick here to the underlying Hopf algebra structures.

Recall first that the set of packed words (or surjections, or ordered partitions of initial subsets of the integers) $\Surj$ is a basis of ${\QSh}$. As a Hopf algebra, ${\QSh}$ is isomorphic to $\WQSym$, the Hopf algebra of word symmetric functions, see e.g. \cite{foissy2} for references  on the subject. This Hopf algebra structure is obtained as follows.
For all $\sigma\in \Surj_k$, $\tau \in \Surj_l$:
$$\sigma\star\tau=\sum_{\zeta \in QSh_{k,l}} \zeta \circ (\sigma \otimes \tau).$$
For all $\sigma \in \Surj_n$:
$$\Delta(\sigma)=\sum_{k=0}^{max(\sigma)} \sigma_{\mid \{1,\ldots,k\}} \otimes Pack(\sigma_{\mid \{k+1,\ldots,max(\sigma)\}}),$$
where for all $I\subseteq \{1,\ldots,max(\sigma)\}$, $\sigma_{\mid I}$ is the packed word obtained by keeping only the letters of $\sigma$ which belong to $I$.

On the other hand, the set of permutations is a basis of the operad $\Sh$. As a Hopf algebra, the latter identifies with the Malvenuto-Reutenauer Hopf algebra \cite{MR} and with the Hopf algebra of free quasi-symmetric functions $\FQSym$. Its Hopf  structure is obtained as follows.
For all $\sigma \in \S_k$, $\tau \in \S_l$:
$$\sigma\star\tau=\sum_{\zeta \in Sh_{k,l}} \zeta \circ (\sigma \otimes \tau).$$
For all $\sigma \in \S_n$:
$$\Delta(\sigma)=\sum_{k=0}^{max(\sigma)} \sigma_{\mid \{1,\ldots,k\}} \otimes Pack(\sigma_{\mid \{k+1,\ldots,max(\sigma)\}}).$$

There is an obvious surjective Hopf algebra epimorphism $\Xi$ from ${\QSh}$ to $\Sh$, sending a packed word $\sigma$ to itself if $\sigma$ is a permutation,
and to $0$ otherwise. From an operadic point of view, this maps amounts to put to zero the $\bullet$ product. There is however another, non operadic, transformation, relating the two structures.

\begin{theo}
We use the following notations:
\begin{enumerate}
\item Let $\sigma \in \S_n$ and $\tau \in \Surj_n$. We shall say that $\tau \propto \sigma$ if:
$$\forall 1\leq i,j \leq n,(\sigma(i)\leq \sigma(j) \Longrightarrow \tau(i)\leq \tau(j)).$$
\item Let $\tau \in \Surj_n$. We put $\displaystyle \tau!=\prod_{i=1}^{max(\tau)} |\tau^{-1}(\{i\})|!$.
\end{enumerate}
We consider the following map:
$$\Phi:\left\{\begin{array}{rcl}
\Sh&\longrightarrow&{\QSh}\\
\sigma \in \S_n&\longrightarrow&\displaystyle \sum_{\tau \propto \sigma} \frac{\tau}{\tau!}.
\end{array}\right.$$
Then $\Phi$ is an injective Hopf algebra morphism. Moreover it is equivariant: for all $\sigma,\tau \in \S_n$,
$$\Phi(\sigma \circ \tau)=\Phi(\sigma)\circ \tau.$$
\end{theo}

\begin{proof} Let $\sigma,\tau \in \S_n$. Then $\tau \propto \sigma$ if, and only if, $\sigma=\tau$. So, for all $\sigma \in \S_n$:
$$\Phi(\sigma)=\sigma+\mbox{linear span of packed words which are not permutations}.$$
So $\Xi \circ \Phi=Id_{\Sh}$, and $\Phi$ is injective. \\

Let $\tau \in \Surj_n$ and $\sigma \in \S_n$. Then $\tau \propto \sigma$ if, and only if, $\tau\circ \sigma^{-1}\propto I_n$.
Moreover, $|\tau \circ \sigma^{-1}|!=\tau!$, as $\sigma$ is a bijection. Hence:
$$\Phi(\sigma)=\sum_{\tau \propto \sigma} \frac{\tau}{\tau!}=\sum_{\rho \propto I_n} \frac{\rho \circ \sigma}{\rho!}=\Phi(I_n)\circ \sigma.$$
More generally, if $\sigma,\tau \in \S_n$, $\Phi(\sigma \circ \tau)=\Phi(I_n)\circ (\sigma \circ \tau)=(\Phi(I_n)\circ \sigma)\circ \tau=\Phi(\sigma)\circ \tau$. \\

Let $\sigma_1 \in \S_{n_1}$ and $\sigma_2 \in \S_{n_2}$.
$$\Phi(\sigma_1)\star\Phi(\sigma_2)=\sum_{\substack {\tau_1\propto \sigma_1,\tau_2\propto\sigma_2 \\ \zeta \in QSh(max(\tau_1),max(\tau_2))}}
\frac{\zeta \circ (\tau_1\otimes \tau_2)}{\tau_1!\tau_2!}.$$
Let $S$ be the set of elements $\sigma \in \Surj_{n_1+n_2}$ such that:
\begin{itemize}
\item For all $1\leq i,j \leq n_1$, $\sigma_1(i)\leq \sigma_1(j)$ $\Longrightarrow$ $\sigma(i)\leq \sigma(j)$.
\item For all $1\leq i,j \leq n_2$, $\sigma_2(i)\leq \sigma_2(j)$ $\Longrightarrow$ $\sigma(i+n_1)\leq \sigma(j+n_2)$.
\end{itemize}

Let $\tau_1\propto \sigma_1$, $\tau_2\propto\sigma_2$ and $\zeta \in QSh(max(\tau_1),max(\tau_2))$. As $\zeta$ is increasing on 
$\{1,\ldots, max(\tau_1)\}$ and $\{max(\tau_1)+1,\ldots,max(\tau_1)+max(\tau_2)\}$, $\zeta \circ (\tau_1\otimes \tau_2)\in S$.
Conversely, if $\sigma \in S$, there exists a unique $\tau_1 \in \Surj_{n_1}$, $\tau_2 \in \Surj_{n_2}$ and $\zeta \in QSh_{max(\tau_1),max(\tau_2)}$
such that $\sigma=\zeta \circ (\tau_1 \otimes \tau_2)$: in particular, $\tau_1=Pack(\sigma(1)\ldots \sigma(n_1))$
and $\tau_2=Pack(\sigma(n_1+1)\ldots \sigma(n_1+n_2))$. As $\sigma \in S$ and $\zeta \in QSh_{max(\tau_1),max(\tau_2)}$,
 $\tau_1\propto \sigma_1$ and $\tau_2 \propto \sigma_2$. Hence:
 $$\Phi(\sigma_1)\star\Phi(\sigma_2)=\sum_{\sigma \in S}\frac{\sigma}{Pack(\sigma(1)\ldots \sigma(n_1))!
 Pack(\sigma(n_1+1)\ldots \sigma(n_1+n_2))!}.$$
 On the other hand:
 $$\Phi(\sigma_1\star\sigma_2)=\sum_{\substack{\zeta \in Sh(n_1,n_2) \\ \tau \propto \zeta \circ (\sigma_1\otimes \sigma_2)}} \frac{\tau}{\tau!}.$$
 Let $\zeta \in Sh(n_1,n_2)$ and $\tau \propto \zeta \circ (\sigma_1\otimes \sigma_2)$. If $1\leq i,j \leq n_1$ and $\sigma_1(i)\leq \sigma_1(j)$, then:
 $$\zeta \circ (\sigma_1 \otimes \sigma_2)(i)=\zeta(\sigma_1(i)) \leq \zeta(\sigma_1(j))=\zeta \circ (\sigma_1 \otimes \sigma_2)(i),$$
 so $\tau(i)\leq \tau(j)$. If $1\leq i,j \leq n_2$ and $\sigma_2(i) \leq \sigma_2(j)$, then:
 $$\zeta \circ (\sigma_1 \otimes \sigma_2)(i+n_1)=\zeta(\sigma_2(i)+max(\sigma_1))
 \leq \zeta(\sigma_2(j)+max(\sigma_1))=\zeta \circ (\sigma_1 \otimes \sigma_2)(j+n_1),$$
 so $\tau(i+n_1)\leq \tau(j+n_2)$. Hence, $\tau \in S$ and finally:
 $$\Phi(\sigma_1\star\sigma_2)=\sum_{\tau \in S} \frac{\tau}{\tau!} \sharp\{\zeta \in Sh(n_1,n_2)\mid
 \tau \propto \zeta \circ (\sigma_1 \otimes \sigma_2) \}.$$
 Let $\tau \in S$. We put $\tau_1=(\tau(1)\ldots \tau(n_1))$ and $\tau_2=(\tau(n_1+1)\ldots \tau(n_1+n_2))$.
 Let $\zeta \in Sh(n_1,n_2)$, such that $\tau \propto \zeta \circ (\sigma_1 \otimes \sigma_2)$. For all $1\leq i\leq max(\tau)$,
$\zeta(\tau^{-1}(\{i\}))=I_i$ is entirely determined and does not depend on $\zeta$. By the increasing conditions on $\zeta$, the determination of such a $\zeta$ consists of
choosing for all $1\leq i\leq max(\tau)$ a bijective map $\zeta_i$ from $\tau^{-1}(\{i\}) $ to $I_i$, such that $\zeta_i$ is increasing
on $\tau^{-1}(\{i\}) \cap \{1,\ldots, n_1\}=\tau_1^{-1}(\{i\})$ and on $\tau^{-1}(\{i\}) \cap \{n_1+1,\ldots, n_1+n_2\}=\tau_2^{-1}(\{i\})$.
Hence, the number of possibilities for $\zeta$ is:
\begin{align*}
&\prod_{i=1}^{max(\tau)} \frac{|\tau^{-1}(i)|!}{|\tau_1^{-1}(\{i\})|!|\tau_2^{-1}(\{i\})|!}\\
&=\frac{\displaystyle \prod_{i=1}^{max(\tau)}|\tau^{-1}(\{i\})|!}{\displaystyle \prod_{i=1}^{max(\tau_1)}|\tau_1^{-1}(\{i\})|!
\prod_{i=1}^{max(\tau_2)}|\tau_2^{-1}(\{i\})|!}\\
&=\frac{\displaystyle \prod_{i=1}^{max(\tau)}|\tau^{-1}(\{i\})|!}
{\displaystyle \prod_{i=1}^{max(Pack(\tau_1))}|Pack(\tau_1)^{-1}(\{i\})|!
\prod_{i=1}^{max(Pack(\tau_2))}|Pack(\tau_2)^{-1}(\{i\})|!}\\
&=\frac{\tau!}{Pack(\tau_1)!Pack(\tau_2)!}.
\end{align*}
Hence:
\begin{align*}
\Phi(\sigma_1\star\sigma_2)&=\sum_{\tau \in S} \frac{\tau}{\tau!} \frac{\tau!}{Pack(\tau(1)\ldots \tau(n_1))!
 Pack(\tau(n_1+1)\ldots \tau(n_1+n_2))!}\\
 &=\Phi(\sigma_1)\star\Phi(\sigma_2).
\end{align*}
So  $\Phi$ is an algebra morphism.  \\

Let $\sigma \in \S_n$. 
\begin{align*}
&\Delta(\Phi(\sigma))\\
&=\sum_{\tau \propto \sigma}\sum_{k=0}^{max(\tau)}
\frac{1}{\tau!} \tau_{\mid\{1,\ldots,k\}}\otimes Pack(\tau_{\mid\{k+1,\ldots,max(\tau)\}}\\
&=\sum_{\tau \propto \sigma}\sum_{k=0}^{max(\tau)}\frac{1}{ \tau_{\mid\{1,\ldots,k\}}! Pack(\tau_{\mid\{k+1,\ldots,max(\tau)\}}!}
\tau_{\mid\{1,\ldots,k\}}\otimes Pack(\tau_{\mid\{k+1,\ldots,max(\tau)\}}\\
&=\sum_{k=0}^n \sum_{\substack{\tau_1\propto \sigma_{\mid \{1,\ldots,k\}}\\ \tau_2 \propto Pack(\sigma_{\mid \{k+1,\ldots,n\}})}}
\frac{\tau_1}{\tau_1!} \otimes \frac{\tau_2}{\tau_2!}\\
&=(\Phi \otimes \Phi)\circ \Delta(\sigma).
\end{align*}
Hence, $\Phi$ is a coalgebra morphism. \end{proof} 
 
{\bf Examples.}
\begin{align*}
\Phi((1))&=(1),\\ 
\Phi((12))&=(12)+\frac{1}{2}(11),\\ 
\Phi((123))&=(123)+\frac{1}{2}(112)+\frac{1}{2}(122)+\frac{1}{6}(111),\\ 
\Phi((1234))&=(1234)+\frac{1}{2}(1123)++\frac{1}{2}(1223)++\frac{1}{2}(1233)\\
&+\frac{1}{4}(1122)+\frac{1}{6}(1112)+\frac{1}{6}(1222)+\frac{1}{24}(1111).
\end{align*}
More generally:
$$\Phi((1\ldots n))=\sum_{k=1}^n \sum_{i_1+\ldots+i_k=n} \frac{1}{i_1!\ldots i_k!} (1^{i_1}\ldots k^{i_k}).$$

{\bf Remark.}  The map $\Phi$ is not a morphism of NSh algebras from $(\Sh,\prec,\succ)$ to $({\QSh},\preceq,\succ)$,
nor to $({\QSh},\prec,\succeq)$. Indeed:
\begin{align*}
\Phi((1)\prec(1))&=(12)+\frac{1}{2}(11),\\
\Phi((1))\prec\Phi((1))&=(12),\\
\Phi((1))\preceq \Phi((1))&=(12)+(11).
\end{align*}

We extend the map $\sigma\longrightarrow F_\sigma$ into a linear map from ${\QSh}$ to $End(T(V))$. By proposition \ref{7}, $F$ is an algebra morphism.

\begin{cor}[Exponential isomorphism]\label{11}
Le us consider the following linear map:
$$\phi:\left\{\begin{array}{rcl}
T(V)&\longrightarrow&T(V)\\
x_1\ldots x_n&\longrightarrow& F_{\Phi(I_n)}(x_1\ldots x_n).
\end{array}\right.$$
Then $\phi$ is a Hopf algebra isomorphism from $(T(V),\shuffle,\Delta)$ to $(T(V),\qshuffle,\Delta)$.
\end{cor}

\begin{proof} Let $x_1,\ldots,x_{k+l} \in V$. 
\begin{align*}
\phi(x_1\ldots x_k \shuffle x_{k+1}\ldots x_{k+l})&=\sum_{\zeta \in Sh(k,l)} F_{\Phi(I_{k+l})}\circ F_\zeta(x_1\ldots x_{k+l})\\
&=\sum_{\zeta \in Sh(k,l)} F_{\Phi(I_{k+l})\circ \zeta}(x_1\ldots x_{k+l})\\
&=\sum_{\zeta \in Sh(k,l)} F_{\Phi(\zeta)}(x_1\ldots x_{k+l})\\
&=F_{\Phi(I_k\star I_l)}(x_1\ldots x_{k+l})\\
&=F_{\Phi(I_k)\star\Phi(I_l)}(x_1\ldots x_{k+l})\\
&=F_{\Phi(I_k)}\qshuffle F_{\Phi(I_l)}(x_1\ldots x_{k+l})\\
&=\sum_{i=0}^{k+l} F_{\Phi(I_k)}(x_1\ldots x_i) \qshuffle F_{\Phi(I_l)}(x_{i+1}\ldots x_{k+l})\\
&= F_{\Phi(I_k)}(x_1\ldots x_k) \qshuffle F_{\Phi(I_l)}(x_{k+1}\ldots x_{k+l})\\
&=\phi(x_1\ldots x_k) \qshuffle \phi(x_{k+1}\ldots x_l).
\end{align*}
So $\phi$ is an algebra morphism.  \\

For any packed words $\sigma \in \Surj_k$, $\tau \in \Surj_l$ 
and all $x_1,\ldots,x_n \in V$
 we define $G_{\sigma \otimes \tau}$ by:
$$G_{\sigma \otimes \tau}(x_1\ldots x_n)= F_\sigma(x_1\ldots x_k) \otimes F_\tau(x_{k+1}\ldots x_{n})
$$
is $k+l=n$ and $=0$ else.
Then, for all increasing packed word $\sigma$, for all $x\in T(V)$:
$$\Delta (F_\sigma(x))=G_{\Delta(\sigma)}(x).$$
Hence, if $x_1,\ldots,x_n \in V$:
\begin{align*}
\Delta \circ \phi(x_1\ldots x_n)&=G_{\Delta(\Phi(I_n))}(x_1\ldots x_n)\\
&=G_{(\Phi\otimes \Phi)\circ \Delta(I_n)}(x_1\ldots x_n)\\
&=\sum_{k=0}^n G_{\Phi(I_k)\otimes \Phi(I_{n-k})}(x_1\ldots x_n)\\
&=\sum_{k=0}^n F_{\Phi(I_k)}(x_1\ldots x_k) \otimes F_{\Phi(I_{n-k})}(x_{k+1}\ldots x_n)\\
&=\sum_{k=0}^n \phi(x_1\ldots x_k) \otimes \phi(x_{k+1}\ldots x_n)\\
&=(\phi \otimes \phi)\circ \Delta(x_1\ldots x_n).
\end{align*}
So $\phi$ is a coalgebra morphism.\\

As the unique bijection appearing in $\Phi(I_n)$ is $I_n$, for all word $x_1\ldots x_n$:
$$\phi(x_1\ldots x_n)=x_1\ldots x_n+\mbox{linear span of words of length}<n.$$
So $\phi$ is a bijection.  \end{proof} 

{\bf Examples.} Let $x_1,x_2,x_3,x_4 \in V$.
\begin{align*}
\phi(x_1)&=x_1,\\
\phi(x_1x_2)&=x_1x_2+\frac{1}{2}x_1.x_2,\\
\phi(x_1x_2x_3)&=x_1x_2x_3+\frac{1}{2}(x_1.x_2)x_3+\frac{1}{2}x_1(x_2.x_3)+\frac{1}{6}x_1.x_2.x_3,\\
\phi(x_1x_2x_3x_4)&=x_1x_2x_3x_4+\frac{1}{2}(x_1.x_2)x_3x_4+\frac{1}{2}x_1(x_2.x_3)x_4\\
&+\frac{1}{2}x_1x_2(x_3.x_4)+\frac{1}{4}(x_1.x_2)(x_3.x_4)+\frac{1}{6}(x_1.x_2.x_3)x_4\\
&+\frac{1}{6}x_1(x_2.x_3.x_4) +\frac{1}{24}x_1.x_2.x_3.x_4.
\end{align*}
More generally, for all $x_1,\ldots, x_n \in V$:
$$\phi(x_1\ldots x_n)=\sum_{k=1}^n \sum_{i_1+\ldots+i_k=n} \frac{1}{i_1!\ldots i_k!}F_{(1^{i_1}\ldots k^{i_k})}(x_1\ldots x_n).$$

{\bf Remarks.} \begin{enumerate}
\item This isomorphism is the morphism denoted by $\exp$ and  obtained in the graded case by Hoffman in \cite{Hoffman}.
\item If $V$ is a trivial algebra, then $\phi=Id_{T(V)}$.
\item This morphism is not a NSh algebra morphism, except if $V$ is a trivial algebra. In fact, except if the product of $V$
is zero, the NSh algebras $(T(V),\preceq,\succ)$ and $(T(V),\prec,\succeq)$ are not commutative, so cannot be isomorphic to a shuffle algebra.
\end{enumerate}

\section{Coalgebra and Hopf algebra endomorphisms}\label{sect:8}

In the previous section, we studied the links between shuffle and quasi-shuffle operads  and obtained as a corollary the exponential isomorphism of Cor. \ref{11} between the shuffle and quasi-shuffle Hopf algebra structures on $T(V)$. This section aims at classifying all such possible (natural, i.e. functorial in commutative algebras $V$) morphisms. We refer to our \cite{foissy2} for applications of natural coalgebra endomorphisms to the study of deformations of shuffle bialgebras.

Recall that we defined $\pi$ as the unique linear endomorphism of the quasi-shuffle bialgebra $T^+(V)$ such that $\pi+\pi \prec Id_{T^+(V)}=Id_{T^+(V)}$.
By proposition \ref{8}, it is equal to $F_{(1)}$, so is the canonical projection on $V$. This construction generalizes as follows. 

Hereafter, we work in the unital setting and write $\varepsilon$  for the canonical projection from $T(V)$ to the scalars (the augmentation map). It behaves as a unit w.r.t. the NQSh products on $End(T^+(V))$: for $g\in End(T^+(V))$, $\varepsilon\prec g=0,\ g\prec \varepsilon=g$.

\begin{prop}\label{12}
Let $f:T(V)\longrightarrow V$ be a linear map such that $f(1)=0$. There exists a unique coalgebra endomorphism $\psi$ of $T(V)$ such that
$\pi \circ \psi=f$. This coalgebra endomorphism is the unique linear endomorphism of $T(V)$ such that $\varepsilon +f\prec \psi=\psi$.
\end{prop}

\begin{proof} {\it First step.} Let us prove the unicity of the coalgebra morphism $\psi$ such that $\pi \circ \psi=f$.
Let $\psi_1,\psi_2$ be two (non zero) coalgebra endomorphisms such that $\pi\circ\psi_1=\pi \circ \psi_2$. Let us prove that for all $x_1,\ldots,x_n \in V$,
$\psi_1(x_1\ldots x_n)=\psi_2(x_1\ldots x_n)$ by induction on $n$. If $n=1$, as $\psi_1(1)$ and $\psi_2(1)$ are both nonzero group-like elements,
they are both equal to $1$. Let us assume the result at all rank $<n$. Then:
\begin{align*}
\Delta\circ \psi_1(x_1\ldots x_n)&=(\psi_1\otimes \psi_1)\circ \Delta(x_1\ldots x_n)\\
&=\psi_1(x_1\ldots x_n)\otimes 1+1\otimes \psi_1(x_1\ldots x_n)\\
&+\sum_{i=1}^{n-1}\psi_1(x_1\ldots x_i)\otimes
\psi_1(x_{i+1}\ldots x_n),\\
\Delta\circ \psi_2(x_1\ldots x_n) &=\psi_2(x_1\ldots x_n)\otimes 1+1\otimes \psi_2(x_1\ldots x_n)\\
&+\sum_{i=1}^{n-1}\psi_2(x_1\ldots x_i)\otimes   \psi_2(x_{i+1}\ldots x_n).
\end{align*}
Applying the induction hypothesis, for all $i\leq 1\leq n-1$, $\psi_1(x_1\ldots x_i)=\psi_2(x_1\ldots x_i)$
and $\psi_1(x_{i+1}\ldots x_n)=\psi_2(x_{i+1}\ldots x_n)$. Consequently, $\psi_1(x_1\ldots x_n)-\psi_2(x_1\ldots x_n)$ is primitive,
so belongs to $V$ and:
$$\psi_1(x_1\ldots x_n)-\psi_2(x_1\ldots x_n)=\pi \circ \psi_1(x_1\ldots x_n)-\pi \circ \psi_2(x_1\ldots x_n)=0.$$

{\it Second step.} Let us prove the existence of a (necessarily unique) endomorphism $\psi$ such that $\psi=\varepsilon +f\prec \psi$. We construct $\psi(x_1\ldots x_n)$ 
for all $x_1,\ldots,x_n \in V$ by induction on $n$ in the following way: $\psi(1)=1$ and, if $n\geq 1$:
$$\psi(x_1\ldots x_n):=f(x_1\ldots f_n)+\sum_{i=1}^{n-1}f(x_1\ldots x_i)\prec \psi(x_{i+1}\ldots x_n).$$
Then $(\varepsilon +f\prec \psi)(1)=\varepsilon(1)=1=\psi(1)$. If $n \geq 1$:
\begin{align*}
&(\varepsilon +f\prec \psi)(x_1\ldots x_n)\\
&=\varepsilon(x_1\ldots x_n)+f(x_1\ldots x_n)+\sum_{i=1}^{n-1} f(x_1\ldots x_i)\prec \psi(x_{i+1}\ldots x_n)\\
&=0+f(x_1\ldots x_n)+\sum_{i=1}^{n-1}  f(x_1\ldots x_i)\prec \psi(x_{i+1}\ldots x_n)\\
&=\psi(x_1\ldots x_n).
\end{align*}
Hence, $\varepsilon +f\prec \psi=\psi$. \\

{\it Third step.} Let $\psi$ such that $\varepsilon +f\prec \psi=\psi$.
Let us prove that $\Delta\circ \psi(x_1\ldots x_n)=(\psi\otimes \psi)\circ \Delta(x_1\ldots x_n)$ by induction on $n$. If $n=0$,
then $\psi(1)=\varepsilon(1)+f(1)=1+0=1$, so $\Delta\circ \psi(1)=(\psi \otimes \psi) \circ \Delta(1)=1\otimes 1$. 
If $n \geq 1$, we put $x=x_1\ldots x_n$, $\Delta(x)=x\otimes 1+1\otimes x+x'\otimes x''$. The induction hypothesis holds for $x''$.
Moreover:
$$\psi(x)=\varepsilon(x)+f(x)+f(x')\prec \psi(x'')=f(x)+f(x')\prec \psi(x'').$$
As $f(x) , f(x')\in  V$ are primitive:
\begin{align*}
\tdelta\circ \psi(x)&=f(x')\otimes \psi(x'')+f(x')\prec \psi(x'')'\otimes \psi(x'')'\\
&=f(x')\otimes \psi(x'')+f(x')\prec \psi(x'')\otimes \psi(x''')\\
&=\psi(x')\otimes \psi(x'')\\
&=(\psi \otimes \psi)\circ \tdelta(x).
\end{align*}
As $\psi(1)=1$, we deduce that $\Delta \circ \psi(x)=(\psi \otimes \psi)\circ \Delta(x)$. So $\psi$ is a coalgebra morphism. 
Moreover, $\pi \circ \psi(1)=\pi(1)=0=f(1)$. If $\varepsilon(x)=0$:
$$\pi\circ \psi(x)=\pi\circ f(x)+\pi(f(x')\prec f(x''))=f(x),$$
as $f(x),(x')\in V$ (so $f(x')\prec f(x'')$ is a linear span of words of length $\geq 2$, so vanishes under the action of $\pi$).
Hence, $\pi \circ \psi=f$. \end{proof}

\begin{prop}
Let $\displaystyle A=\sum_{n\geq 1} a_n X^n$ be a formal series without constant term. Let $f_A$ be the linear map from $T(V)$ to $V$ defined
by $f_A(x_1\ldots x_n)=a_n x_1\bullet \ldots \bullet x_n$ and let $\phi_A$ be the unique coalgebra endomorphism of $T(V)$ such that
$\pi\circ \phi_A=f_A$. For all $x_1,\ldots,x_n \in V$:
\begin{equation}
\label{E29}\phi_A(x_1\ldots x_n)=\sum_{k=1}^n \sum_{i_1+\ldots+i_k=n} a_{i_1}\ldots a_{i_k}
F_{(1^{i_1}\ldots k^{i_k})}(x_1\ldots x_n).
\end{equation}\end{prop}

\begin{proof} Note that $f_A(x_1\ldots x_n)=a_n F_{(1^n)}(x_1\ldots x_n)$. Let $\phi$ be the morphism defined by the second member
of (\ref{E29}). Then $(\varepsilon +f_A\prec \phi)(1)=1+f_A(1)=1=\phi(1)$. If $n \geq 1$:
\begin{align*}
&(\varepsilon +f_A\prec \phi)(x_1\ldots x_n)\\
&=f_A(x_1\ldots x_n)+\sum_{i=1}^{n-1} f_A(x_1\ldots x_i)\prec \phi(x_{i+1}\ldots x_n)\\
&=a_n F_{(1^n)}(x_1\ldots x_n)\\
&+\sum_{i=1}^{n-1}\sum_{k=2}^n \sum_{i_2+\ldots+i_k=n-i}a_ia_{i_2}\ldots a_{i_k} F_{(1^i)}\prec F_{(1^{i_2}\ldots (k-1)^{i_k})}(x_1\ldots x_n)\\
&=a_n F_{(1^n)}(x_1\ldots x_n)\\
&+\sum_{i=1}^{n-1}\sum_{k=2}^n \sum_{i+i_2+\ldots+i_k=n}a_ia_{i_2}\ldots a_{i_k} \prec F_{(1^i2^{i_2}\ldots k^{i_k})}(x_1\ldots x_n)\\
&=\phi(x_1\ldots x_n).
\end{align*}
By unicity in proposition \ref{12}, $\phi=\phi_A$. \end{proof} 

{\bf Remark.} The morphism $\phi$ defined in corollary \ref{11} is $\phi_{exp(X)-1}$.

\begin{prop}
$\phi_X=Id$ and for all formal series $A,B$ without constant terms, $\phi_A\circ \phi_B=\phi_{A\circ B}$.
\end{prop}

\begin{proof} For all $x_1,\ldots, x_n \in V$, $\pi \circ Id(x_1\ldots x_n)=\delta_{1,n} x_1\ldots x_n=f_X(x_1\ldots x_n)$.
By unicity in proposition \ref{12}, $\phi_X=Id$. Moreover:
\begin{align*}
&\pi \circ \phi_A\circ \phi_B(x_1\ldots x_n)\\
&=f_A\left(\sum_{k=1}^n \sum_{i_1+\ldots+i_k=n}\hspace{-4mm} b_{i_1}\ldots b_{i_k} (x_1\bullet \ldots \bullet x_{i_1})\ldots 
(x_{i_1+\ldots+i_{k-1}+1}\bullet \ldots \bullet x_{1+\ldots+i_k})\right)\\
&=\sum_{k=1}^n \sum_{i_1+\ldots+i_k=n}a_k b_{i_1}\ldots b_{i_k}x_1\bullet \ldots \bullet x_n\\
&=f_{A\circ B}(x_1\ldots x_n).
\end{align*}
By unicity in proposition \ref{12}, $\phi_A\circ \phi_B=\phi_{A\circ B}$. \end{proof} 

So the set of all $\phi_A$, where $A$ is a formal series such that $A(0)=0$ and $A'(0)\neq 1$, is a subgroup of the group of coalgebra isomorphisms of $T(V)$,
isomorphic to the group of formal diffeomorphisms of the line.

\begin{cor}
The inverse of the isomorphism $\phi$ defined in corollary \ref{11} is $\phi_{ln(1+X)}$:
$$\phi^{-1}(x_1\ldots x_n)=\sum_{k=1}^n \sum_{i_1+\ldots+i_k=n}\frac{(-1)^{n+k}}{i_1\ldots i_k}
F_{(1^{i_1}\ldots k^{i_k})}(x_1\ldots x_n).$$
\end{cor}

\begin{prop}\label{classif} Let $A\in K[[X]]^+$. 
\begin{enumerate}
\item $\phi_A:(T(V),\shuffle,\Delta)\longrightarrow (T(V),\shuffle,\Delta)$ is a Hopf algebra morphism for any commutative algebra $V$ if,  
and only if, $A=aX$ for a certain $a\in K$.
\item $\phi_A:(T(V),\shuffle,\Delta)\longrightarrow (T(V),\qshuffle,\Delta)$ is a Hopf algebra morphism for any commutative algebra $V$ if,  
and only if, $A=exp(aX)-1$ for a certain $a\in K$.
\item $\phi_A:(T(V),\qshuffle,\Delta)\longrightarrow (T(V),\qshuffle,\Delta)$ is a Hopf algebra morphism for any commutative algebra $V$ if,  
and only if, $A=(1+X)^a-1$ for a certain $a\in K$.
\item $\phi_A:(T(V),\qshuffle,\Delta)\longrightarrow (T(V),\shuffle,\Delta)$ is a Hopf algebra morphism for any commutative algebra $V$ if,  
and only if, $A=a\ln(1+X)$ for a certain $a\in K$.
\end{enumerate}\end{prop}

\begin{proof} First, note that for any $x_1,\ldots,x_k \in V$:
$$\pi \circ \phi_A(x_1\ldots x_k)=a_kF_{(1\ldots 1)}(x_1\ldots x_k).$$
Consequently, for any commutative algebra $V$, for any $x,x_1,\ldots,x_k \in V$, $k\geq 1$:
\begin{align*}
\pi \circ \phi_A(x \shuffle x_1\ldots x_k)&=\pi(xx_1\ldots x_{k+1}+\ldots+x_1\ldots x_{k+1}x)\\
&=(k+1)a_{k+1}x.x_1\cdot\ \ldots\ \cdot x_k,\\
\pi(\phi_A(x)\shuffle \phi_A(x_1\ldots x_k))&=0,\\
\pi(\phi_A(x)\qshuffle \phi_A(x_1\ldots x_k))&=a_1a_kx.x_1\cdot \ldots \ \cdot x_k.
\end{align*}

1. We assume that $\phi_A$ is an algebra morphism for any $V$ for the shuffle product.
Let us choose an algebra $V$ and elements $x,x_1,\ldots,x_k \in V$ such that $x.x_1\cdot \ \ldots \ \cdot x_k \neq 0$ in $V$. 
As $\phi(x\shuffle x_1\ldots x_k)=\phi(x)\shuffle \phi(x_1\ldots x_k)$, applying $\pi$, we deduce that for all $k\geq 1$,
$(k+1)a_{k+1}=0$, so $a_{k+1}=0$. Hence, $A=a_1X$. Conversely, for any $x_1,\ldots,x_k \in V$,
$\phi_{aX}(x_1\ldots x_k)=a_1^k x_1\ldots x_k$, so $\phi_{aX}$ is an endomorphism of the Hopf algebra $(T(V),\shuffle,\Delta)$.\\

2. We already proved that $\phi_{exp(X)-1}$ is a Hopf algebra morphism from $(T(V),\shuffle,\Delta)$ to $(T(V),\qshuffle,\Delta)$. 
By composition:
$$\phi_{exp(aX)-1}=\phi_{exp(X)-1}\circ \phi_{aX}:(T(V),\shuffle,\Delta)\longrightarrow (T(V),\shuffle,\Delta)\longrightarrow (T(V),\qshuffle,\Delta)$$
is a Hopf algebra morphism.

We assume that $\phi_A$ is an algebra morphism for any $V$ from the shuffle product to the quasi-shuffle product. Let us choose an algebra $V$,
and $x,x_1,\ldots,x_k \in V$, such that $x.x_1\cdot\ \ldots \ \cdot x_k \neq 0$ in $V$. 
As $\phi(x\shuffle x_1\ldots x_k)=\phi(x)\qshuffle \phi(x_1\ldots x_k)$, applying $\pi$, we deduce that for all $k\geq 1$,
$(k+1)a_{k+1}=a_1a_k$, so $a_k=\frac{a_1^k}{k!}$ for all $k\geq 1$. Hence, $A=exp(a_1X)-1$.\\

3. The following conditions are equivalent:
\begin{itemize}
\item For any $V$, $\phi_A:(T(V),\qshuffle,\Delta)\longrightarrow (T(V),\qshuffle,\Delta)$ is a Hopf algebra morphism.
\item For any $V$, $\phi_{\ln(1+X)}\circ\phi_A\circ \phi_{exp(X)-1}:(T(V),\shuffle,\Delta)\longrightarrow (T(V),\shuffle,\Delta)$ is a Hopf algebra morphism.
For any $V$, $\phi_{\ln(1+X)\circ A\circ (exp(X)-1)}:(T(V),\shuffle,\Delta)\longrightarrow (T(V),\shuffle,\Delta)$ is a Hopf algebra morphism.
\item There exists $a\in K$, $\ln(1+X)\circ A\circ (exp(X)-1)=aX$.
\item There exists $a\in K$, $A=(1+X)^a-1$.
 \end{itemize}

4. Similar proof. \end{proof}

{\bf Remark.} The Proposition \ref{classif} classifies actually all the Hopf algebra endomorphisms and morphisms relating shuffle and quasi-shuffle algebras $T(V)$, that are natural (i.e. functorial) in $V$. This naturality property follows formally from the study of nonlinear Schur-Weyl duality in \cite{novelli2,foissy2}.

\section{Coderivations and graduations}\label{sect:9}
The present section complements the previous one that studied coalgebra endomorphisms. We aim at investigating here coderivations of quasi-shuffle bialgebras. As an application we recover the existence of a natural graded structure on the Hopf algebras $(T(V),\qshuffle,\Delta)$ \cite{foissy2}.

{\bf Notations.} Let $A$ be a NQSh algebra, $f \in End_K(A)$ and $v \in A$. We define:
$$f\prec v:\left\{\begin{array}{rcl}
A&\longrightarrow&A\\
x&\longrightarrow&f(x)\prec v,
\end{array}\right. \hspace{1cm}
v\prec f:\left\{\begin{array}{rcl}
A&\longrightarrow&A\\
x&\longrightarrow&v\prec f(x).
\end{array}\right.$$

\begin{prop}\label{17}
Let $f:T(V)\longrightarrow V$ be a linear map. There exists a unique coderivation $D$ of $T(V)$ such that $\pi \circ D=f$.
Moreover, $D$ is the unique linear endomorphism of $T(V)$ such that $D=f+\pi \prec D+f \prec Id$.
\end{prop}

\begin{proof} {\it First step.} Let us prove that the unicity of the coderivation $D$ such that $\pi \circ D=f$. Let $D_1$ and $D_2$ be two coderivations
such that $\pi \circ D_1=\pi \circ D_2$. Let us prove that $D_1(x_1\ldots x_n)=D_2(x_1\ldots x_n)$ by induction on $n$. 
$$\Delta\circ D_1(1)=(D_1\otimes Id+Id \otimes D_1)(1\otimes 1)=D_1(1)\otimes 1+1\otimes D_1(1),$$
so $D_1(1)\in Prim(T(V))=V$. Similarly, $D_2(1)\in V$. Hence, $D_1(1)=\pi \circ D_1(1)=\pi \circ D_2(1)=D_2(1)$.
Let us assume the result at all ranks $<n$.  If $p=1$ or $2$:
$$\Delta \circ D_p(x_1\ldots x_n)=\sum_{i=0}^n D_p(x_1\ldots x_i)\otimes x_{i+1}\ldots x_n+
\sum_{i=0}^n x_1\ldots x_i \otimes D_p(x_{i+1}\ldots x_n).$$
Applying the induction hypothesis at all ranks $<k$, we obtain by substraction:
$$\Delta\circ (D_1-D_2)(x_1\ldots x_n)=(D_1-D_2)(x_1\ldots x_n)\otimes 1+1\otimes (D_1-D_2)(x_1\ldots x_n).$$
So $(D_1-D_2)(x_1\ldots x_n)\in V$. Applying $\pi$:
$$(D_1-D_2)(x_1\ldots x_n)=\pi \circ (D_1-D_2)(x_1\ldots x_n)=0.$$
So $D_1(x_1\ldots x_n)=D_2(x_1\ldots x_n)$. \\

{\it Second step.}  Let us prove the existence of a map $D$ such that $D=f+\pi \prec D+f \prec Id$. 
We define $D(x_1\ldots x_n)$ by induction on $n$ by $D(1)=f(1)$ and:
$$D(x_1\ldots x_n)=x_1\prec D(x_2\ldots x_n)+\sum_{i=0}^{n-1} f(x_1\ldots x_i) \prec x_{i+1}\ldots x_n+f(x_1\ldots x_n).$$
Then $(f+\pi \prec D+f \prec Id)(1)=f(1)=D(1)$. If $n \geq 1$:
\begin{align*}
&(f+\pi \prec D+f \prec Id)(x_1\ldots x_n)\\
&=f(x_1\ldots x_n)+\sum_{i=1}^{n} \pi(x_1\ldots x_i)\prec D(x_{i+1}\ldots x_n)\\
&+\sum_{i=0}^{n-1} f(x_1\ldots x_i)\prec x_{i+1}\ldots x_n\\
&=f(x_1\ldots x_n)+x_1\prec D(x_2\ldots x_n)+\sum_{i=0}^{n-1} f(x_1\ldots x_i) \prec x_{i+1}\ldots x_n\\
&=D(x_1\ldots x_n).
\end{align*}
So $D=f+\pi \prec D+f \prec Id$.\\

{\it Last step.} Let $D$ such that $D=f+\pi \prec D+f \prec Id$. Let us prove that $\Delta\circ D(x_1\ldots x_n)=(D\otimes Id+Id \otimes D)\circ \Delta(x_1\ldots x_n)$ by induction on $n$. If $n=0$:
\begin{align*}
\Delta\circ D(1)&=\Delta(f(1))\\
&=f(1)\otimes 1+1\otimes f(1)\\
&=D(1)\otimes 1+1\otimes D(1)\\
&=(D\otimes Id+Id\otimes D)(1\otimes 1).
\end{align*}
Let us assume the result at all ranks $<n$. 
\begin{align*}
D(x_1\ldots x_n)&=(f+\pi \prec D+f \prec Id)(x_1\ldots x_n)\\
&=\sum_{i=1}^{n} \pi(x_1\ldots x_i)\prec D(x_{i+1}\ldots x_n)+\sum_{i=0}^{n-1} f(x_1\ldots x_i)\prec x_{i+1}\ldots x_n\\
&+f(x_1\ldots x_n)\\
&=x_1 D(x_2\ldots x_n)+\sum_{i=0}^n f(x_1\ldots x_i) x_{i+1}\ldots x_n.
\end{align*}
Hence:
\begin{align*}
&\Delta\circ D(x_1\ldots x_n))\\
&=\sum_{j=1}^n x_1D(x_2\ldots x_j)\otimes x_{j+1}\ldots x_n\\
&+\sum_{j=1}^n x_1\ldots x_j \otimes D(x_{j+1}\ldots x_n)+1\otimes x_1D(x_2\ldots x_n)\\
&+\sum_{i=0}^n\sum_{j=i}^n f(x_1\ldots x_i)x_{i+1}\ldots x_j \otimes x_{j+1}\ldots x_n\\
&+\sum_{i=0}^n 1\otimes f(x_1\ldots x_i)x_{i+1}\ldots x_n\\
&=\sum_{j=1}^n x_1D(x_2\ldots x_j)\otimes x_{j+1}\ldots x_n\\
&+\sum_{j=1}^n x_1\ldots x_j \otimes D(x_{j+1}\ldots x_n)+1\otimes x_1D(x_2\ldots x_n)\\
&+\sum_{j=1}^n\sum_{i=1}^j f(x_1\ldots x_i)x_{i+1}\ldots x_j \otimes x_{j+1}\ldots x_n\\
&+f(1)\otimes x_1\ldots x_n+\sum_{i=0}^n 1\otimes f(x_1\ldots x_i)x_{i+1}\ldots x_n\\
&=\sum_{j=0}^n D(x_1\ldots x_j) \otimes x_{j+1}\ldots x_n+\sum_{j=1}^n x_1\ldots x_j \otimes D(x_{j+1}\ldots x_n)\\
&=(D\otimes Id+Id \otimes D)\circ \Delta(x_1\ldots x_n).
\end{align*}
Moreover, $\pi \circ D(1)=\pi \circ f(1)=f(1)$; if $n\geq 1$:
\begin{align*}
 \circ D(x_1\ldots x_n)&=\pi(x_1D(x_2\ldots x_n))+\sum_{i=0}^n \pi(f(x_1\ldots x_i) x_{i+1}\ldots x_n)\\
 &=0+f(x_1\ldots x_n).
\end{align*}
So $\pi \circ D=f$. \end{proof}

\begin{prop}
Let $\displaystyle A=\sum_{n\geq 1} a_n X^n$ be a formal series without constant term. Let $D_A$ be the unique coderivation of $T(V)$ such that
$\pi\circ \phi_A=f_A$. For all $x_1,\ldots,x_n \in V$:
\begin{equation}
\label{E30}D_A(x_1\ldots x_n)=\sum_{i=1}^n a_i \sum_{j=1}^{n-i+1}F_{(12\ldots j-1 j^i j+1\ldots n-i+1)}(x_1\ldots x_n).
\end{equation}\end{prop}

\begin{proof}  Let $D$ be the linear endomorphism defined by the right side of (\ref{E30}).  As $f_A(1)=0$, we get by induction on $n$:
\begin{align*}
&(f+\pi \prec D+f \prec Id)(x_1\ldots x_n)\\
&=f(x_1\ldots x_n)+x_1D(x_2\ldots x_n)+\sum_{i=1}^{n-1} f(x_1\ldots x_i)x_{i+1}\ldots x_n\\
&=x_1D(x_2\ldots x_n)+\sum_{i=1}^n f(x_1\ldots x_i)x_{i+1}\ldots x_n\\
&=\sum_{i=1}^{n-1}a_i \sum_{j=2}^{n-i+1}F_{(12\ldots j-1 j^i j+1\ldots n-i+1)}(x_1\ldots x_n)
+\sum_{i=1}^n a_i F_{(1^i 2\ldots n-i+1)}(x_1\ldots x_n)\\
&=\sum_{i=1}^n a_i \sum_{j=1}^{n-i+1}F_{(12\ldots j-1 j^i j+1\ldots n-i+1)}(x_1\ldots x_n)\\
&=D(x_1\ldots x_n).
\end{align*}
Moreover, $\pi \circ D(x_1\ldots x_n)=a_n x_1\bullet \ldots \bullet x_n=f_A(x_1\ldots x_n)$. The unicity in proposition \ref{17} implies that
$D=D_A$. \end{proof}

\begin{cor}
For all word $x_1\ldots x_n$, $D_X(x_1\ldots x_n)=n x_1\ldots x_n$.
\end{cor}

\begin{proof} Indeed, $\displaystyle D_X(x_1\ldots x_n)=\sum_{j=1}^nF_{(12\ldots j-1 j^1 j+1\ldots n)}(x_1\ldots x_n)=n x_1\ldots x_n$. \end{proof} 

{\bf Remark.} Let $A$ and $B$ be two formal series and $\lambda \in K$. As $D_A+\lambda D_B$ is a coderivation and 
$\pi \circ (D_A+\lambda D_B)=f_A+\lambda f_B=f_{A+\lambda B}$:
$$D_A+\lambda D_B=D_{A+\lambda B}.$$
Moreover, the group of coalgebra automorphims of $T(V)$ acts on the space of coderivations of $T(V)$ by conjugacy. Let us precise this action 
if we work only with automorphisms and coderivations associated to formal series.

\begin{prop} 
Let $A,B$ be two formal series without constant terms, such that $A'(0)\neq 0$. Then:
$$\phi_A^{-1}\circ D_B\circ \phi_A=D_{\frac{ B \circ A}{A'}}.$$
\end{prop}

\begin{proof} By linearity and continuity of the action, it is enough to prove this formula if $B=X^p$. We denote by $C$ the inverse of $A$ for the composition.
\begin{align*}
&\pi \circ \phi_A^{-1} \circ D_{X^p}\circ \phi_A(x_1\ldots x_n)\\
&=f_C\circ D_{X_p}\left(\sum_{k=1}^n \sum_{i_1+\ldots+i_k=n} a_{i_1}\ldots a_{i_k} 
F_{(1^{i_1}\ldots k^{i_k})}(x_1\ldots x_n)\right)\\
&=\sum_{k=p-1}^n \sum_{i_1+\ldots+i_k=n} (k-p-1) c_{k-p+1}a_{i_1}\ldots a_{i_k} 
x_1\bullet \ldots \bullet x_n.
\end{align*}
So $ \pi \circ \phi_{A^{-1}} \circ D_{X^p}\circ \phi_A$ is the linear map associated to the formal series:
\begin{align*}
\left(\sum_{k=p-1}^\infty (k-p+1)c_{k-p+1}X^k\right)\circ A&=\left(\sum_{i=0}^\infty ia_i X^{i-1+p}\right)\circ A\\
&=(X^p C')\circ A\\
&=A^p C'\circ A\\
&=\frac{A^p}{A'}.
\end{align*}
Hence, $\phi_{A^{-1}}\circ D_{X^p}\circ \phi_A=D_{\frac{A^p}{A'}}$. \end{proof}

\begin{cor}
The eigenspaces of the coderivation $D_{(1+X)ln(1+X)}$ give a gradation of the Hopf algebra $(T(V),\qshuffle,\Delta)$.
\end{cor}

\begin{proof} Let $D=\phi \circ D_X \circ \phi^{-1}$. As $\phi=\phi_{exp(X)-1}$:
$$D=\phi_{ln(1+X)}^{-1}\circ D_X\circ \phi_{ln(1+X)}=D_{(1+X)ln(1+X)}.$$
As $D_X$ is a derivation of the algebra $(T(V), \shuffle)$ and $\phi$ is an algebra isomorphism from $(T(V),\shuffle)$ to $(T(V),\qshuffle)$,
$D$ is is a derivation of the algebra $(T(V),\qshuffle)$. As it is conjugated to $D_X$, its eigenvalues are the elements of $\mathbb{N}$. \end{proof} \ \\

{\bf Remark.} As $\displaystyle (1+X)ln(1+X)=1+\sum_{k=2}^\infty \frac{(-1)^k}{k(k-1)} X^k$:
\begin{align*}
&D_{(1+X)ln(1+X)}(x_1\ldots x_n)\\
&=n x_1\ldots x_n+\sum_{i=2}^n \sum_{j=1}^{n-i+1} \frac{(-1)^i}{i(i-1)} x_1\ldots x_{j-1}(x_j\bullet \ldots \bullet x_{j+i-1})x_{j+i}\ldots x_n.
\end{align*}
The gradation of $A=(T(V),\qshuffle)$ is given by:
$$A_n=Vect\left(\begin{array}{c}
\displaystyle \sum_{k=1}^n
\sum_{i_1+\ldots+i_k=n} \frac{1}{i_1!\ldots i_k!} \left(\prod_{i=1}^{i_1}x_i\right)\ldots \left(\prod_{i=i_1+\ldots+i_{k-1}+1}^{i_1+\ldots+i_k }x_i\right),\\
x_1,\ldots,x_n \in V\end{array}\right).$$

\section{Decorated operads and graded structures}\label{sect:10}

In many applications, algebras over operads carry a natural graduation. This is because geometrical objects (polynomial vector fields, spaces, differential forms\dots ), but also combinatorial and algebraic ones carry often a graduation (or a dimension, a cardinal\dots ) that is better taken into account in the associated algebra structures.
As far as quasi-shuffle algebras are concerned, they carry often naturally a graduation in their application domains : think to quasi-symmetric functions and multizeta values (MZVs) \cite{cartier2002fonctions}; Ecalle's mould calculus and dynamical systems \cite{fauvet2012ecalle}; iterated integrals of It\^o type in stochastic calculus \cite{ebrahimi2015flows,ebrahimi2015exponential}.

Here, we explain briefly how the formalism of operads can be adapted to take into account graduations. We detail then the case of quasi-shuffle algebras and conclude by studying the analogue, in this context, of the classical descent algebra of a graded commutative or cocommutative Hopf algebra \cite{patras1994algebre}.

In this section, we denote by $A=\bigoplus\limits_{n\in\mathbf N}A_n$ (where $A_0=k$, the ground field), a graded, connected, quasi-shuffle bialgebra. By graded we mean that all the structure maps ($\prec,\bullet,\Delta$) are graded maps. Then $Prim(A)=V=\bigoplus\limits_{n\in\mathbf N^\ast}V_n$ is an associative, commutative graded algebra for the product 
$\bullet$ and we can identify $A$ and the quasi-shuffle algebra $T^+(V)$ as graded tridendriform algebras. Be aware that the graduation of $T^+(V)$ is not the tensor length: for example, for $v_1\in V_{n_1},\dots ,v_k\in V_{n_k}$, the degree of the tensor $v_1\dots v_k\in V^{\otimes k}$ is now $n_1+\dots +n_k$.

It is an easy exercice to adapt the definition of operads to the graded case: whereas the component ${\mathbf F}_n$ of an operad identifies with the set of multilinear elements in the $n$ letters $x_1,\dots,x_n$ in the free algebra $F(X_n)$, $X_n:=\{x_1,\dots ,x_n\}$, the corresponding graded operad ${\mathbf F}_n^d$ is obtained by considering the set of multilinear elements in the free algebra $F(X_n^d)$, where $X_n^d:=\{x_i^k\}_{i\leq n,d\in\mathbf{N}^\ast}$ and 
where multilinear means now that we consider the subspace $\bigcap\limits_{i\leq n}\Gamma_i$, where $\Gamma_i$ is the eigenspace of $F(X_n^d)$ associated to the eigenvalue $\lambda$ of the map induced by $x_i^d\to\lambda x_i^d,\ x_j^d\to x_j^d$ else.
We call ${\mathbf F}^d=\cup_n{\mathbf F}_n^d$ the (integer-)decorated operad associated to $\mathcal F$-algebras. 

The decorated operad $\QSh^d$ is then spanned by decorated packed words, where:
\begin{defi}
A \emph{decorated packed word} of length $k$ is a pair $(\sigma,d)$, where $\sigma$ is a packed word of length $k$ and $d$ is a map from $\{1,\ldots,k\}$ into
$\mathbb{N}^*$. We denote it by $\left(\begin{array}{ccc} \sigma(1)&\ldots&\sigma(k)\\
d(1)&\ldots&d(k)\end{array}\right)$.
\end{defi}

{\bf Notation.} Let $(\sigma,d)=\left(\begin{array}{ccc} \sigma(1)&\ldots&\sigma(k)\\d(1)&\ldots&d(k)\end{array}\right)$ be a decorated 
packed word. Let $m$ be the maximum of $\sigma$. We define $F_{(\sigma,d)} \in End_K(A)$ in the following way: for all $x_1,\ldots,x_l \in V$, homogeneous,
$$F_{(\sigma,d)}(x_1\ldots x_l)=\left\{
\begin{array}{l}
\displaystyle \left(\prod_{\sigma(i)=1}x_i\right)\ldots \left(\prod_{\sigma(i)=m}x_i\right) 
\begin{array}{rcl}
\mbox{ if }k&=&l\mbox{ and }\\
deg(x_1)&=&d(1),\\
&\vdots&\\
deg(x_k)&=&d(k),\end{array}\\[9mm]
0\mbox{ otherwise}.
\end{array}\right.$$
Note that in each parenthesis, the product is the product $\bullet$ of $V$. For example, if $x,y,z\in V$ are homogeneous,
$$F_{\left(\begin{array}{ccc}2&1&2\\a&b&c\end{array}\right)}(xyz)=y(x\bullet z)$$
if $deg(x)=a$, $deg(y)=b$, and $deg(z)=c$, and $0$ otherwise. \\

The subspace of $End_K(A)$ generated by these maps is stable under composition and the noncommutative quasi-shuffle products:

\begin{prop}\label{23}
Let 
$$(\sigma,d)=\left(\begin{array}{ccc} \sigma(1)&\ldots&\sigma(k)\\d(1)&\ldots&d(k)\end{array}\right)\mbox{ and }
(\tau,e)=\left(\begin{array}{ccc} \tau(1)&\ldots&\tau(l)\\e(1)&\ldots&e(l)\end{array}\right)$$
 be two decorated packed words.
 $max(\tau)=k$ and for all $1\leq j\leq k$, $\displaystyle \sum_{\tau(i)=j}e(i)=d(j)$, then:
$$F_{(\sigma,d)} \circ F_{(\tau,e)}=F_{\scriptsize{\left(\begin{array}{ccc}
\sigma \circ \tau(1)&\ldots&\sigma\circ \tau(l)\\e(1)&\ldots&e(l)\end{array}\right)}}.$$
Otherwise, this composition is equal to $0$. Moreover:
\begin{align*}
&F_{(\sigma,d)}\prec F_{(\tau,e)}\\
&=\sum_{\substack{Pack(u(1)\ldots u(k))=\sigma,\\ Pack(u(k+1)\ldots u(k+l))=\tau,\\min(u(1)\ldots u(k))<min(u(k+1)\ldots u(k+l))}}
F_{\scriptsize{\left(\begin{array}{cccccc}
u(1)&\ldots&u(k)&u(k+1)&\ldots&u(k+l)\\
d(1)&\ldots&d(k)&e(1)&\ldots&e(l)\end{array}\right)}},\\
&F_{(\sigma,d)}\succ F_{(\tau,e)}\\
&=\sum_{\substack{Pack(u(1)\ldots u(k))=\sigma,\\ Pack(u(k+1)\ldots u(k+l))=\tau,\\min(u(1)\ldots u(k))>min(u(k+1)\ldots u(k+l))}}
F_{\scriptsize{\left(\begin{array}{cccccc}
u(1)&\ldots&u(k)&u(k+1)&\ldots&u(k+l)\\
d(1)&\ldots&d(k)&e(1)&\ldots&e(l)\end{array}\right)}},\\
&F_{(\sigma,d)}\bullet F_{(\tau,e)}\\
&=\sum_{\substack{Pack(u(1)\ldots u(k))=\sigma,\\ Pack(u(k+1)\ldots u(k+l))=\tau,\\min(u(1)\ldots u(k))=min(u(k+1)\ldots u(k+l))}}
F_{\scriptsize{\left(\begin{array}{cccccc}
u(1)&\ldots&u(k)&u(k+1)&\ldots&u(k+l)\\
d(1)&\ldots&d(k)&e(1)&\ldots&e(l)\end{array}\right)}}.
\end{align*}\end{prop}

\begin{proof} Direct computations. \end{proof} 

{\bf Remarks.} \begin{enumerate}
\item For all packed word $(\sigma(1)\ldots \sigma(n))$:
$$F_{(\sigma(1)\ldots \sigma(n))}=\sum_{d(1),\ldots,d(n)\geq 1}
F_{\left(\begin{array}{ccc}\sigma(1)&\ldots&\sigma(n)\\d(1)&\ldots&d(n)\end{array}\right)}.$$
\item In general, this action of decorated packed words is not faithful. For example, if $V=K[X]_+$, where $X$ is homogeneous of degree $n$,
then $F_{\left(\begin{array}{cc}1&2\\1&1\end{array}\right)} =F_{\left(\begin{array}{cc}2&1\\1&1\end{array}\right)}$.
Indeed, both sends the word $XX$ on itself and all the other words on $0$.
\item Here is an example where the action is faithful. Let $V=K[X_i\mid i\geq 1]_+$, where $X_i$ is homogeneous of degree $1$ for all $i$.
Let us assume that $\sum a_{(\sigma,d)}F_{(\sigma,d)}=0$. Acting on the word $(X_1^{a_1})\ldots (X_k^{a_k})$, we obtain:
$$\sum_{length(\sigma)=k}a_{\left(\begin{array}{ccc}
\sigma(1)&\ldots&\sigma(k)\\a_1&\ldots&a_k\end{array}\right)}
\left(\prod_{\sigma(i)=1} X_i^{a_i}\right)\ldots \left(\prod_{\sigma(i)=max(\sigma)}X_i^{a_i}\right)=0.$$
As the $X_i$ are algebraically independent, the words appearing in this sum are linearly independent, so for all $(\sigma,d)$, $a_{(\sigma,d)}=0$. \\
\end{enumerate}

{\bf Notations.} \begin{enumerate}
\item For all $n \geq 1$, we put:
$$p_n=\sum_{k=1}^n \sum_{d(1)+\ldots+d(k)=n} F_{\scriptsize{\left(\begin{array}{ccc}1&\ldots&k\\d(1)&\ldots&d(k)\end{array}\right)}}.$$
The map $p_n$ is the projection on the space of words of degree $n$, so $\displaystyle \sum_{n\geq 1}p_n=Id_A$.
\item For all $n \geq 1$, we put:
$$q_n=F_{\scriptsize{\left(\begin{array}{c}1\\n\end{array}\right)}}.$$
\end{enumerate}

The map $q_n$ is the projection  on the space of letters of degree $n$, so, by proposition \ref{8}, 
$\displaystyle q=\sum_{n\geq 1}q_n=F_{(1)}$ is the projection $\pi$ of proposition \ref{4}.
It is not difficult to deduce, in the same way as proposition 12 of \cite{Foissy}, the following result:

\begin{theo}\label{qshdesc}
The NQSh subalgebra $QDesc(A)$ of $End_K(A)$ generated by the homogeneous components $p_n$ of $Id_A$ 
is also generated by the homogeneous components $q_n$ of the projection on $Prim(A)$ of proposition \ref{4}. Moroever, for all $n\geq 1$:
\begin{align*}
q_n&=\sum_{k=1}^n (-1)^{k+1}\sum_{a_1+\ldots+a_k=n} p_{a_1}\prec (p_{a_2}\qshuffle \ldots \qshuffle p_{a_k}).
\end{align*}\end{theo}

{\bf Remark.} This result is the quasi-shuffle analog of the statement that the descent algebra of a graded connected cocommutative Hopf algebra $H$  (the convolution subalgebra of $End(H)$ generated by the graded projections) is equivalently generated by the graded components of the convolution logarithm of the identity \cite{patras1994algebre}.

\section{Freeness of the decorated quasi-shuffle operad}\label{sect:11}

In this section, we show that the decorated quasi-shuffle operad $\QSh^d$ is free as a NSh algebra using the bidendriform techniques developed in \cite{foissy2007bidendriform}.

We denote by ${\QSh}^d_+$ the subspace of the decorated quasi-shuffle operad generated by nonempty decorated packed words. As for a well-chosen graded quasi-shuffle bialgebra $A$
the action of packed words is faithful, we deduce that ${\QSh}^d_+$ inherits a NQSh algebra structure by:
\begin{align*}
&(\sigma,d)\prec (\tau,e)\\
&=\sum_{\substack{Pack(u(1)\ldots u(k))=\sigma,\\ Pack(u(k+1)\ldots u(k+l))=\tau,\\min(u(1)\ldots u(k))<min(u(k+1)\ldots u(k+l))}}
\scriptsize{\left(\begin{array}{cccccc}
u(1)&\ldots&u(k)&u(k+1)&\ldots&u(k+l)\\
d(1)&\ldots&d(k)&e(1)&\ldots&e(l)\end{array}\right)},\\
&(\sigma,d)\succ (\tau,e)\\
&=\sum_{\substack{Pack(u(1)\ldots u(k))=\sigma,\\ Pack(u(k+1)\ldots u(k+l))=\tau,\\min(u(1)\ldots u(k))>min(u(k+1)\ldots u(k+l))}}
\scriptsize{\left(\begin{array}{cccccc}
u(1)&\ldots&u(k)&u(k+1)&\ldots&u(k+l)\\
d(1)&\ldots&d(k)&e(1)&\ldots&e(l)\end{array}\right)},\\
&(\sigma,d)\bullet (\tau,e)\\
&=\sum_{\substack{Pack(u(1)\ldots u(k))=\sigma,\\ Pack(u(k+1)\ldots u(k+l))=\tau,\\min(u(1)\ldots u(k))=min(u(k+1)\ldots u(k+l))}}
\scriptsize{\left(\begin{array}{cccccc}
u(1)&\ldots&u(k)&u(k+1)&\ldots&u(k+l)\\
d(1)&\ldots&d(k)&e(1)&\ldots&e(l)\end{array}\right)}.
\end{align*}

{\bf Notations.} Let $(\sigma,d)$ be a decorated packed word of length $k$ and let $I\subseteq \{1,\ldots,max(\sigma)\}$.
We put $\sigma^{-1}(I)=\{i_1,\ldots,i_l\}$, with $i_1<\ldots<i_l$. The decorated packed word $(\sigma,d)_{\mid I}$ is
$(Pack(\sigma(i_1),\ldots,\sigma(i_l)),(d(i_1),\ldots,d(i_l)))$.

\begin{defi}
We define two coproducts on ${\QSh}^d_+$ in the following way: for all nonempty packed word $(\sigma,d)$,
\begin{align*}
\Delta_\prec(\sigma,d)&=\sum_{i=\sigma(1)}^{max(\sigma)-1} (\sigma,d)_{\mid \{1,\ldots,i\}} \otimes (\sigma,d)_{\mid
\{\{i+1,\ldots,max(\sigma)\}},\\
\Delta_\succ(\sigma,d)&=\sum_ {i=1}^{\sigma(1)-1} (\sigma,d)_{\mid \{1,\ldots,i\}} \otimes (\sigma,d)_{\mid
\{\{i+1,\ldots,max(\sigma)\}}.
\end{align*}
Then ${\QSh}^d_+$ is a NSh coalgebra, that is to say:
\begin{align}
(\Delta_\prec \otimes Id) \circ \Delta_\prec&=(Id \otimes (\Delta_\prec+\Delta_\succ))\circ \Delta_\prec,\\
(\Delta_\succ \otimes Id)\circ \Delta_\prec&=(Id \otimes \Delta_\prec) \circ \Delta_\succ,\\
((\Delta_\prec+\Delta_\succ)\otimes Id) \circ \Delta_\succ&=(Id \otimes \Delta_\succ)\circ \Delta_\succ.
\end{align}
For all $a,b \in {\QSh}^d_+$:
\begin{align}
\label{E38} \Delta_\prec(a\prec b)&=a'_\prec \prec b'\otimes a''_\prec \star b''+a'_\prec \prec b\otimes a''_\prec+a'_\prec\otimes a''_\prec \star b\\
\nonumber &+a\prec b'\otimes b''+a\otimes b,\\
\Delta_\prec(a\succ b)&=a'_\prec \succ b'\otimes a''_\prec \star b''+a\succ b'\otimes b''+a'_\prec \succ b \otimes a''_\prec,\\
\Delta_\prec(a\bullet b)&=a'_\prec \bullet b'\otimes a''_\prec \star b''+a'_\prec \bullet b\otimes a''_\prec+a\bullet b'\otimes b'',\\
\Delta_\succ(a\prec b)&=a'_\succ \prec b'\otimes a''_\succ \star b''+a'_\succ \prec b\otimes a''_\succ+a'_\succ \otimes a''_\succ \star b,\\
\Delta_\succ(a\succ b)&=a'_\succ \succ b''\otimes a''_\succ \star b''+a'_\succ \succ b \otimes a''_\succ+b'_\succ \otimes a\star b''+b \otimes a,\\
\label{E43}\Delta_\succ (a\bullet b)&=a'_\succ \bullet b'\otimes a''_\succ \star b''+a'_\succ \bullet b \otimes a''_\succ.
\end{align}\end{defi}

\begin{proof} Let $(\sigma,d)$ be a decorated packed word. Then:
\begin{align*}
&(\Delta_\prec \otimes Id) \circ \Delta_\prec(\sigma,d)=(Id \otimes (\Delta_\prec+\Delta_\succ))\circ \Delta_\prec(\sigma,d)\\
&=\sum_{\sigma(1)\leq i<j\leq max(\sigma)-1}
(\sigma,d)_{\mid \{1,\ldots,i\}}\otimes (\sigma,d)_{\mid \{i+1,\ldots,j\}}\otimes (\sigma,d)_{\mid\{j+1,\ldots,max(\sigma)\}},\\ \\
&(\Delta_\succ \otimes Id)\circ \Delta_\prec(\sigma,d)=(Id \otimes \Delta_\prec) \circ \Delta_\succ(\sigma,d)\\
&=\sum_{1\leq i<\sigma(1)\leq j\leq max(\sigma)-1}
(\sigma,d)_{\mid \{1,\ldots,i\}}\otimes (\sigma,d)_{\mid \{i+1,\ldots,j\}}\otimes (\sigma,d)_{\mid\{j+1,\ldots,max(\sigma)\}},\\ \\
&((\Delta_\prec+\Delta_\succ)\otimes Id) \circ \Delta_\succ(\sigma,d)=(Id \otimes \Delta_\succ)\circ \Delta_\succ(\sigma,d)\\
&=\sum_{1\leq i<j<\sigma(1)}
(\sigma,d)_{\mid \{1,\ldots,i\}}\otimes (\sigma,d)_{\mid \{i+1,\ldots,j\}}\otimes (\sigma,d)_{\mid\{j+1,\ldots,max(\sigma)\}}.
\end{align*}

Let us prove (\ref{E38}), for $a=(\sigma,d)$ and $b=(\tau,e)$ two decorated packed words of respective length $k$ and $l$. We put:
$$a\otimes b=\left(\begin{array}{cccccc}
\sigma(1)&\ldots&\sigma(k)&\tau(1)+max(\sigma)&\ldots&\tau(l)+max(\tau)\\
d(1)&\ldots&d(k)&e(1)&\ldots&e(l)
\end{array}\right).$$
 Then $a\prec b$ is the sum of all decorated packed words obtained by quasi-shuffling in all possible ways the values of the letters in the first row of $a\otimes b$, in such a way 
that $1$ occurs only in the first $k$ columns; $\Delta_\prec(a\otimes b)$ is then given by separating the letters of the first row of these decorated packed words
in such a way that the first letter appears in the left side. So at least one of the $k$ first letters appears on the left side. This gives five possible cases:
\begin{enumerate}
\item All the $k$ first letters are on the left and all the $l$ last letters are on the right. Necessarily, this case comes from the decorated packed word $a\otimes b$,
and this gives the term $a\otimes b$.
\item All the $k$ first letters are on the left and at least one of the $l$ last letters is on the left. This gives the term $a\prec b'\otimes b''$.
\item At least one of the $k$ first letters is on the right and all the $l$ last letters are on the left. This gives the term $a'_\prec \prec b \otimes a''_\prec$.
\item At least one of the $k$ first letters is on the right and all the $l$ last letters are on the right. This gives the term $a'_\prec \otimes a''_\prec \star b$.
\item At least one of the $k$ first letters is on the right and there are some of the $l$ last letters on both sides. This gives the term $a'_\prec \prec b'\otimes a''_\prec \star b''$.
\end{enumerate}
Summing all these terms, we obtain (\ref{E38}). The other compatibilities can be proved similarly. \end{proof}

{\bf Remark.} We also obtain, by addition:
\begin{align}
\Delta_\prec(a\preceq b)&=a'_\prec \preceq b'\otimes a''_\prec \star b''+a'_\prec \preceq b\otimes a''_\prec+a'_\prec\otimes a''_\prec \star b\\
\nonumber&+a\preceq b'\otimes b''+a\otimes b,\\
\Delta_\prec(a\succeq b)&=a'_\prec \succeq b'\otimes a''_\prec \star b''+a\succeq b'\otimes b''+a'_\prec \succeq b \otimes a''_\prec,\\
\Delta_\succ(a\preceq b)&=a'_\succ \preceq b'\otimes a''_\succ \star b''+a'_\succ \preceq b\otimes a''_\succ+a'_\succ \otimes a''_\succ \star b,\\
\Delta_\succ(a\succeq b)&=a'_\succ \succeq b''\otimes a''_\succ \star b''+a'_\succ \succeq b \otimes a''_\succ+b'_\succ \otimes a\star b''+b \otimes a;\\
\nonumber \\
\tdelta(a \prec b)&=a'\prec b'\otimes a''\star b''+a'\prec b\otimes a''+a'\otimes a''\star b\\
\nonumber&+a\prec b'\otimes b''+a\otimes b,\\
\tdelta(a\succ b)&=a'\succ b'\otimes a''\star b''+a'\succ b\otimes a''+a\succ b'\otimes b''\\
\nonumber&+b'\otimes a\star b''+b\otimes a,\\
\tdelta(a\bullet b)&=a'\bullet b'\otimes a''\star b''+a'\bullet b\otimes a''+a\bullet b'\otimes b'';\\
\nonumber \\
\tdelta(a \preceq b)&=a'\preceq b'\otimes a''\star b''+a'\preceq b\otimes a''+a'\otimes a''\star b\\
\nonumber&+a\preceq b'\otimes b''+a\otimes b,\\
\tdelta(a\succeq b)&=a'\succeq b'\otimes a''\star b''+a'\succeq b\otimes a''+a\succeq b'\otimes b''\\
\nonumber&+b'\otimes a\star b''+b\otimes a.
\end{align}
Consequently, $({\QSh}^d_+,\succ^{op},\preceq^{op},\Delta_\succ^{op},\Delta_\prec^{op})$
and $({\QSh}^d_+,\succeq^{op},\prec^{op},\Delta_\succ^{op},\Delta_\prec^{op})$ are bidendriform bialgebras \cite{foissy2007bidendriform}. By the bidendriform rigidity theorem,
$({\QSh}^d_+,\preceq,\succ)$ and $({\QSh}^d_+,\prec,\succeq)$ are free NSh algebras.
Forgetting the decoration, we get back theorem 2.5 of \cite{Novelli}, up to a permutation of maximum and minimum, and first and last letters.\\

Forgeting the decorations, we obtain a NQSh algebra structure on ${\QSh}_+$ and a NSh coalgebra structure, with compatibilities
(\ref{E38})-(\ref{E43}). 
Let us describe, for completeness sake, the dual (half-)products and coproducts. The elements of the dual basis of packed words are denoted by $N_u$.

\begin{prop}\begin{enumerate}
\item For all nonempty packed words $\sigma,\tau$, of respective lengths $k$ and $l$:
\begin{align*}
N_\sigma \prec N_\tau&=\sum_{\alpha \in Sh_{k,l}^\prec} N_{(\sigma\otimes \tau) \circ \alpha^{-1}},&
N_\sigma \succ N_\tau&=\sum_{\alpha \in Sh_{k,l}^\succ} N_{(\sigma\otimes \tau) \circ \alpha^{-1}}.
\end{align*}
\item For any nonempty packed word $\sigma$ of length $n$, denoting by $f(\sigma)$ the index of the first appearance of $1$ in $\sigma$ 
and by $l(\sigma)$ the index of the last appearance of $1$ in $\sigma$:
\begin{align*}
\tdelta_\prec(N_\sigma)&=\sum_{k=l(\sigma)}^{n-1} N_{pack(\sigma(1)\ldots \sigma(k))}\otimes N_{pack(\sigma(k+1)\ldots \sigma(n))},\\
\tdelta_\succ(N_\sigma)&=\sum_{k=1}^{f(\sigma)-1} N_{pack(\sigma(1)\ldots \sigma(k))}\otimes N_{pack(\sigma(k+1)\ldots \sigma(n))},\\
\tdelta_\bullet(N_\sigma)&=\sum_{k=f(\sigma)}^{l(\sigma)-1} N_{pack(\sigma(1)\ldots \sigma(k))}\otimes N_{pack(\sigma(k+1)\ldots \sigma(n))}.
\end{align*}\end{enumerate}\end{prop}

\section{The quasi-shuffle analog of the descent algebra}\label{sect:12}

Recall that, given a graded NQSh bialgebra $A$, we introduced $QDesc(A)$, the quasi-shuffle analogue of the descent algebra defined as the NQSh subalgebra of $End(A)$ generated by the graded projections or, equivalently, by the graded components of the projection on $Prim(A)$. 
We write $\QDesc$ for the corresponding NQSh subalgebra of ${\QSh}^d$ (the subalgebra generated by the ${1}\choose{d}$).

Recall first some properties of NSh algebras.\\

{\bf Notations.} Let $n \geq 1$.
\begin{enumerate}
\item \begin{enumerate}
\item  Let $\TSchroder(n)$ be the set of Schr\"oder trees of degree $n$, that is to say reduced planar rooted trees with $n+1$ leaves. 
\item For any set $D$, let $\TSchroder^D(n)$ be the set of reduced planar rooted trees $t$ with $n+1$ leaves, 
such that the $n$ spaces between the leaves of $t$ are decorated by elements of $D$. 
\item $\displaystyle \TSchroder^D=\bigsqcup_{n\geq 1}\TSchroder^D(n)$.
\end{enumerate}
\item Let $t_1,\ldots,t_k\in \TSchroder^{\N^*}$ and let $d_1,\ldots,d_{k-1} \in \N^*$. The element $t_1\vee_{d_1}\ldots \vee_{d_{k-1}} t_k$
is obtained by grafting $t_1,\ldots,t_k$ on a common root;
for all $1\leq i \leq k$, the space between the right leaf of $t_i$ and the left leaf of $t_{i+1}$ is decorated by $d_i$. 
\end{enumerate}

Following \cite{Loday}, $\TSchroder^D$ is a basis of the free NQSh algebra generated by $D$, $NQSh(D)$.
The three products are inductively defined: if $t=t_1\vee_{d_1}\ldots \vee_{d_{k-1}} t_k$ and $t'=t'_1\vee_{d'_1}\ldots \vee_{d'_{l-1}} t'_l \in \TSchroder(D)$, then
\begin{align*}
t\succ t'&=(t\star t'_1) \vee_{d'_1}t'_2 \vee_{d'_2} \ldots \vee_{d'_{l-1}}t'_l,\\
t\prec t'&=t_1\vee_{d_1}\ldots \vee t_{k-1} \vee_{d_{k-2}} \ldots \vee_{d_{k-1}}(t_k\star t'),\\
t\bullet t'&=t_1\vee_{d_1}\ldots \vee_{d_{k-1}} (t_k\star t'_1) \vee_{d'_1}\ldots \vee_{d'_{l-1}} t'_l.
\end{align*}

Sending any non binary tree to $0$, we obtain the free NSh algebra $NSh(D)$ generated by $D$.
A basis is given by the set of planar binary trees $\Tbin(D)\subseteq \TSchroder(D)$ whose spaces between the leaves are decorated 
by elements of $D$. The products are given in the following way: if  $t=t_1 \vee_d t_2$ and $t'=t'_1 \vee_{d'} t'_2$, then:
\begin{align*}
t\succ t'&=(t\star t'_1)\vee_{d'} t'_2,\\
t\prec t'&=t_1 \vee_d (t_2 \star t').
\end{align*}

We denote by $NQSh(1)$ and by $NSh(1)$ the free NQSh and the free NSh algebra on one generator. 
The set $\TSchroder$ is a basis of $NQSh(1)$, and $\Tbin$ is a basis of $NSh(1)$.\\

{\bf Examples.}
\begin{align*}
\TSchroder(0)&=\Tbin(0)=\{\bun\},&
\TSchroder(1)&=\Tbin(1)=\{\bdeux\},\\
\TSchroder(2)&=\left\{\btroisun,\btroisdeux,\schtrois\right\},&
\Tbin(2)&=\left\{\btroisun,\btroisdeux\right\},\\
\TSchroder(3)&=\left\{\begin{array}{c}\bquatreun,\bquatredeux,\bquatretrois,\bquatrequatre\\
\bquatrecinq,\schquatreun,\schquatredeux,\schquatretrois,\\
\schquatrequatre,\schquatrecinq,\schquatresix \end{array}\right\},&
\Tbin(3)&=\left\{\bquatreun,\bquatredeux,\bquatretrois,\bquatrequatre,\bquatrecinq\right\}.
\end{align*}

We define now inductively a surjective map $\varrho$ from the set of  packed words decorated by $D$ into $\TSchroder^D$ in the following way:
\begin{enumerate}
\item $\varrho(1)=\bun$.
\item If $w=(\sigma,d)$, let $\sigma^{-1}(1)=\{i_1,\ldots,i_k\}$, $i_1<\ldots<i_k$.
We put:
\begin{align*}
w_1&=Pack\left(
\begin{array}{ccc}
\sigma(1)&\ldots&\sigma(i_1-1)\\
d(1)&\ldots&d(i_1-1)\end{array}\right),\\
w_2&=Pack\left(\begin{array}{ccc}
\sigma(i_1+1)&\ldots&\sigma(i_2-1)\\
d(i_1+1)&\ldots&d(i_2-1)\end{array}\right),\\
&\vdots&\\
w_{k+1}&=Pack\left(\begin{array}{ccc}
\sigma(i_k+1)&\ldots &\sigma(n)\\
d(i_k+1)&\ldots&d(n)
\end{array}\right). \end{align*}
Then:
$$\varrho(\sigma,d)=\varrho(w_1)\vee_{d(i_1)}\ldots\vee_{d(i_k)}\varrho(w_{k+1}).$$
\end{enumerate}

If $w=(\sigma,d)$ is a decorated packed word of length $n$, $\varrho(w)$ is an element of $\TSchroder^D(n)$ such that the spaces between the leaves
are decorated from left to right by $d(1),\ldots,d(n)$. In particular $\varrho{1\choose d}$ is the  tree $\bdeux$ d-decorated. \\

For any $t\in \TSchroder^{\N^*}$, we put:
$$\Omega(t)=\sum_{\sigma\in \Surj, \varrho(\sigma)=t}\sigma \in {\QSh}^d_+.$$
We extend $\Omega:NQSh(\N^*)\longrightarrow {\QSh}^d_+$ by linearity map. It is clearly injective.\\

{\bf Examples.} 
\begin{align*}
\Omega(\bdeux)&=(1),&
\Omega(\btroisun)&=(21),&
\Omega(\btroisdeux)&=(12),\\
\Omega(\schtrois)&=(11),&
\Omega(\bquatreun)&=(321),&
\Omega(\bquatredeux)&=(231),\\
\Omega(\bquatretrois)&=(132),&
\Omega(\bquatrequatre)&=(123),&
\Omega(\bquatrecinq)&=(212)+(312)+(213),\\
\Omega(\schquatreun)&=(221),&
\Omega(\schquatredeux)&=(211),&
\Omega(\schquatretrois)&=(121),\\
\Omega(\schquatrequatre)&=(112),&
\Omega(\schquatrecinq)&=(122),&
\Omega(\schquatresix)&=(111).
\end{align*}

\begin{theo}
The map $\Omega$ is an injective morphism of NQSh algebras. Consequently, $\QDesc$,
the NQSh subalgebra of ${\QSh}^d_+$ generated by the elements $\binom{1}{d}$, $d\geq 1$, is free and isomorphic to $NQSh({\N}^\ast )$.
\end{theo}

\begin{proof}  Let $w=(\sigma,d)$  be a packed word of length $n$ and let $i_1,\ldots,i_k$ be integers such that $i_1+\ldots+i_k=n$.
For all $d_1,\ldots,d_{k-1}\geq 1$, we put:
\begin{align*}
&ins_{i_1,\ldots,i_k}^{d_1,\ldots,d_{k-1}}(w)\\
&=\scriptsize{\left(\begin{array}{ccccccccc}
\sigma(1)+1&\ldots&\sigma(i_1)+1&1&\ldots&1&\sigma(i_1+\ldots+i_{k-1}+1)+1&\ldots&\sigma(n)+1\\
d(1)&\ldots&d(i_1)&d_1&\ldots&d_{k-1}&d(i_1+\ldots+i_{k-1}+1)&\ldots &d(n)
\end{array}\right)}.
\end{align*}

It is not difficult to show that:
$$\Omega(t_1\vee_{d_1}\ldots \vee_{d_{k-1}}t_k)
=ins_{|t_1|,\ldots,|t_k|}^{d_1,\ldots,d_{k-1}}(\Omega(t_1)\star\ldots \star \Omega(t_k)).$$
Hence, if $t=t_1\vee_{d_1}\ldots \vee_{d_{k-1}} t_k$ and $t'=t'_1\vee_{d'_1}\ldots \vee_{d'_{l-1}} t'_l$:
\begin{align*}
\Omega(t)\succ \Omega(t')&=ins_{|t|+|t'_1|,\ldots, |t'_l|}^{d'_1,\ldots,d'_{l-1}}(\Omega(t) \star \Omega(t'_1)\star\ldots \star \Omega(t'_l)),\\
\Omega(t)\prec \Omega(t')&=ins_{|t_1|,\ldots, |t_k|+|t|}^{d_1,\ldots,d_{k-1}}(\Omega(t_1) \star \ldots \star \Omega(t_k) \star \Omega(t')),\\
\Omega(t)\bullet \Omega(t')&=ins_{[t_1|,\ldots,|t_k|+|t'_1|,\ldots,|t'_l|}^{d_1,\ldots,d_{k-1},d'_1,\ldots,d'_{l-1}}
(\Omega(t_1)\star \ldots \star \Omega(t_k)\star \Omega(t'_1)\star \ldots \star \Omega(t'_l)).
\end{align*}
An induction on $m+n$ proves that for
$t\in \TSchroder^{\N^*}(m)$, $t'\in \TSchroder^{\N^*}(n)$:
\begin{align*}
\Omega(t\succ t')&=\Omega(t)\succ \Omega(t'),&
\Omega(t\prec t')&=\Omega(t)\prec \Omega(t'),&
\Omega(t\bullet t')&=\Omega(t)\bullet \Omega(t').
\end{align*}
 So $\Omega$ is an injective morphism of NQSh algebras. \end{proof}

\section{Lie theory, continued}\label{sect:13}

In classical Lie theory, it has been realized progressively that many applications of the combinatorial part of the theory rely on the freeness of the Malvenuto-Reutenauer algebra of permutations (for us, the operad $\Sh$ or, equivalently, the  algebra of free quasi-symmetric functions $\FQSym$) as a noncommutative shuffle bialgebra (and more precisely, as a bidendriform bialgebra \cite{foissy2007bidendriform}). 
As such, $\Sh$ has two remarquable subalgebras. The first is $\PBT$, the noncommutative shuffle sub-bialgebra freely generated as a noncommutative shuffle algebra by the identity permutation in $\S_1$ (in particular $\PBT$ is isomorphic to $NSH(1)$, the free NQSh algebra on one generator). Its elements can be understood as linear combinations of planar binary trees ($\PBT$ can be constructed directly as a subspace of the direct sum of the symmetric group algebras is by using a construction going back to Viennot: a natural partition of the symmetric groups parametrized by planar binary trees), see \cite{hivert2005algebra,hivert2008trees,loday1998hopf}. The second, $\Desc$, is known as the descent algebra \cite{Reutenauer}, is isomorphic to ${\mathbf{Sym}}$, the Hopf algebra of noncommutative symmetric functions, and is the sub Hopf algebra of $\PBT$ and $\Sh$ freely generated as an associative algebra by (all) the identity permutations using the convolution product $\star$. We get:
$$\Desc={\mathbf{Sym}} \subset \PBT =NSH(1)\subset \Sh=\FQSym.$$

The situation is similar when moving to surjections, that is to ${\QSh}$. As we already saw, the noncommutative quasi-shuffle sub-bialgebra freely generated by the identity permutation in $\S_1$ (i.e. the packed word $1$) is the free NQS algebra on one generator, identified with $\ST$, the linear span of Schr\"oder trees. 
The sub Hopf algebra of $\ST$ and ${\QSh}$ freely generated as an associative algebra by (all) the identity permutations using the convolution product $\star$ is isomorphic (using e.g. that it is a free associative algebra over a countable set of generators) to $\Desc$.
We get:
$$\Desc={\mathbf{Sym}}\subset \ST =NQSH(1)\subset {\QSh}=\WQSym.$$

The aim of the present and last section is to compare explicitely the two sequences of inclusions.
The existence of a Hopf algebra map from $\Sh=\FQSym$ to ${\QSh}=\WQSym$ was obtained in \cite[Cor. 18]{foissy2}. 
The existence of a map comparing the two copies of the descent algebra follows, a simple direct proof was given in 
\cite[Lemma 7.1]{ebrahimi2015flows}.
We aim here at refining these results and extend the constructions to planar and Schr\"oder trees.

We start by showing how planar trees ($\PBT$) can be embedded into Schr\"oder trees ($\ST$).
\begin{defi}
Let $t,t'\in \TSchroder$. 
\begin{enumerate}
\item We denote by $R(t)$ the set of internal edges of $t$ which are right, that is to say edges $e$ such that:
\begin{itemize}
\item both extremities of $e$ are internal vertices.
\item $e$ is the edge which is at most on the right among all the egdes with the same origin as $e$.
\end{itemize}
\item Let $I\subseteq R(T)$. We denote by $t/I$ the planar reduced tree obtained by contracting all the edges $e\in I$.
\item We shall say that $t'\leq t$ if there exists $I\subseteq R(t)$, such that $t'=t/I$.
\end{enumerate}\end{defi}

{\bf Remarks.} If $I\subseteq R(t)$, then $R(t/I)=R(t)\setminus I$. Moreover, if $I,J\subseteq R(t)$ are disjoint, then $(t/I)/J=t/(I\sqcup J)$.
This implies that $\leq$ is a partial order on $\TSchroder$.\\

{\bf Examples.} Here are the Hasse graphs of $\TSchroder(2)$ and $\TSchroder(3)$.
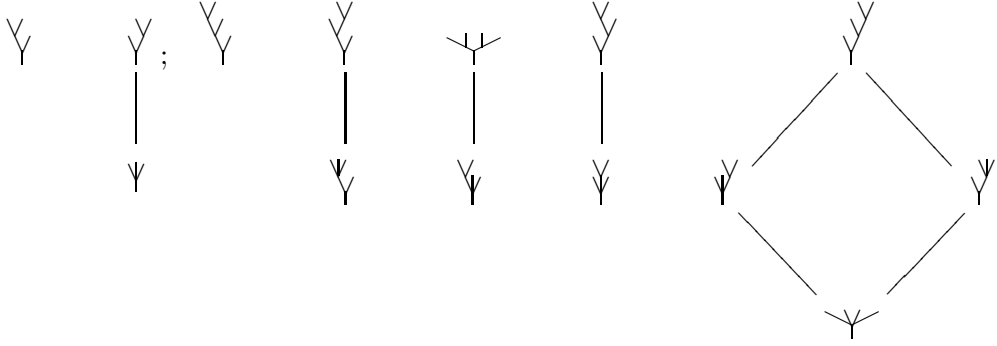
\begin{align*}
&\xymatrix{\btroisun&\btroisdeux \ar@{-}[d]\\
&\schtrois};&
&\xymatrix{\bquatreun&\bquatredeux \ar@{-}[d]&\bquatrecinq \ar@{-}[d]&\bquatretrois \ar@{-}[d]&&\bquatrequatre \ar@{-}[dl] \ar@{-}[dr]\\
&\schquatreun&\schquatredeux&\schquatretrois&\schquatrequatre \ar@{-}[dr]&&\schquatrecinq\ar@{-}[dl]\\
&&&&&\schquatresix}
\end{align*}

It is possible to prove the following points:
\begin{itemize}
\item For any $t\in \TSchroder$, there exists a unique $b(t)\in \Tbin$, such that $t\leq b(t)$. 
We denote by $I(t)$ the unique subset $I\subseteq R(b(t))$, such that $t=b(t)/I$.
\item  For any $t,t'\in \TSchroder$, $t\leq t'$ if, and only if, $b(t)=b(t')$ and $I(t)\supseteq I(t')$.
\end{itemize}

\begin{theo}
The following map is an injective morphism of bidendriform bialgebras:
$$\psi:\left\{\begin{array}{rcl}
(\PBT ,\prec,\succ,\Delta_\prec,\Delta_\succ)&\longrightarrow&(\ST ,\preceq,\succ,\Delta_\prec,\Delta_\succ)\\
t\in \Tbin&\longrightarrow&\displaystyle \sum_{t'\leq t} t'.
\end{array}\right.$$
\end{theo} 

\begin{proof} By universal properties of free objects, there exists a unique morphism of noncommutative shuffle algebras $\psi'$ from $(NSh(1)=\PBT ,\prec,\succ)$ to $(NQSh(1)=\ST ,\preceq,\succ)$,
sending $\bdeux$ to $\bdeux$. As $\bdeux$ is a primitive element (in the bidendriform sense) for both sides, $\psi'$ is a morphism of bidendriform bialgebras.
We shall prove that $\psi=\psi'$. 

Let us show that for all $t_1,t_2\in \Tbin$, 
$$\psi'(t_1\vee t_2)=\psi'(t_1)\succ \bdeux \preceq \psi'(t_2).$$
and 
$$\psi(t_1\vee t_2)=\psi(t_1)\succ \bdeux \preceq \psi(t_2).$$
The identity $\psi=\psi'$ will follow by induction.

The identity involving $\psi'$ follows immediately from the identity, in $\Tbin$:
$$t_1\vee t_2=t_1 \succ \bdeux \prec t_2.$$

Let us consider the action of $\psi$. We put $t=t_1\vee t_2$. We first consider the case where $t_2=\bun$. In this case, $R(t)=R(t_1)$ and for any $I\subseteq R(t_1)$, $t/I=(t_1/I)\vee \bun$.
Hence:
$$\psi(t)=\sum_{I\subseteq R(t_1)} (t_1/I)\vee \bun=\left(\sum_{I\subseteq R(t_1)} t_1/I \right)\succ \bdeux=\psi(t_1)\succ \bdeux \preceq \bun.$$
We now consider the case where $t_2\neq \bun$. Let $r$ be the internal edge of $t$ relating the root of $t$ to the root of $t_2$.
Then $R(t)=R(t_1)\sqcup R(t_2)\sqcup \{r\}$. Let $I_1\subseteq R(t_1)$, $I_2\subseteq R(t_2)$. Then:
$$t/I_1\sqcup I_2=(t_1/I_1)\vee (t_2/I_2)=(t_1/I_1)\succ \bdeux \prec (t_2/I_2).$$
We put $t_2/i_2=t_3\vee \ldots \vee t_k$. Then:
\begin{align*}
t/I_1\sqcup I_2 \sqcup \{r\}&=t_1/I_1\vee t_3\vee \ldots \vee t_k\\
&=(t_1/I_1\vee \bun)\bullet (t_3\vee \ldots \vee t_k)\\
&=((t_1/I_1) \succ \bdeux) \bullet (t_2/I_2)\\
&=(t_1/I_1)\succ \bdeux \bullet  (t_2/I_2).
\end{align*}
Hence:
\begin{align*}
\psi(t)&=\sum_{I_1\subseteq R(t_1),I_2\subseteq R(t_2)}(t_1/I_1)\succ \bdeux \prec (t_2/I_2)+(t_1/I_1)\succ \bdeux\bullet (t_2/I_2)\\
&=\sum_{I_1\subseteq R(t_1),I_2\subseteq R(t_2)}(t_1/I_1)\succ \bdeux \preceq (t_2/I_2)\\
&=\psi(t_1)\succ \bdeux \preceq \psi(t_2).
\end{align*}
So $\psi=\psi'$. As $\leq$ is an order, $\psi$ is injective. \end{proof}

We investigate now how the injection of $\PBT$ into $\ST$ behaves with respect to the respective embeddings into $\Sh$ and ${\QSh}$.
We consider the morphism:
$$\Omega:\left\{\begin{array}{rcl}
\ST =NQSh(1)&\longrightarrow&{\QSh}\\
t&\longrightarrow&\displaystyle \sum_{\sigma,\varrho(\sigma)=t} \sigma.
\end{array}\right.$$ 
There exists a unique map from $\PBT =NSh(1)$ to $\mathcal{S}h$, denoted by $\Omega'$, making the following diagram commuting:
$$\xymatrix{\ST\ar[r]^\Omega \ar@{>>}[d]&{\QSh}\ar@{->>}[d]\\
\PBT\ar[r]_{\Omega'}&\mathcal{S}h}$$
where the vertical arrows are the canonical projection. For any $t\in \Tbin$:
$$\Omega'(t)=\sum_{\sigma\in \mathfrak{S},\varrho(\sigma)=t} \sigma.$$

{\bf Examples.} 
\begin{align*}
\Omega'(\bdeux)&=(1),&
\Omega'(\btroisun)&=(21),&
\Omega'(\btroisdeux)&=(12),&
\Omega'(\bquatreun)&=(321),\\
\Omega'(\bquatredeux)&=(231),&
\Omega'(\bquatretrois)&=(132),&
\Omega'(\bquatrequatre)&=(123),&
\Omega'(\bquatrecinq)&=(312)+(213)
\end{align*}

\begin{prop}\cite{FoissyMalvenuto}
Let $\sigma,\tau$ be two packed words of the same length $n$. We shall say that $\sigma\leq \tau$ if:
\begin{enumerate}
\item If $i,j \in [n]$ and $\sigma(i) \leq \sigma(j)$, then $\tau(i)\leq \tau(j)$.
\item If $i,j \in [n]$, $i<j$ and $\sigma(i)>\sigma(j)$, then $\tau(i)>\tau(j)$.
\end{enumerate} 
Then $\leq$ is a partial order. Moreover, the following map is a Hopf algebra morphism:
$$\Psi:\left\{\begin{array}{rcl}
\Sh&\longrightarrow&{\QSh}\\
\sigma&\longrightarrow&\displaystyle \sum_{\tau\leq \sigma} \tau.
\end{array}\right.$$
\end{prop}

Here are the Hasse graphs of $Surj_2$ and $Surj_3$:
$$\xymatrix{(12)\ar@{-}[d]&(21)\\ (11)};$$
$$\xymatrix{&(123)\ar@{-}[rd] \ar@{-}[ld] &\\(122)\ar@{-}[rd] &&(112)\ar@{-}[ld]\\&(111)&}\hspace{.5cm}
\xymatrix{(132)\ar@{-}[d]\\ (121)}\hspace{.5cm}\xymatrix{(213)\ar@{-}[d]\\ (212)}\hspace{.5cm}
\xymatrix{(231)\ar@{-}[d]\\ (221)}\hspace{.5cm}\xymatrix{(312)\ar@{-}[d]\\ (211)}\hspace{.5cm}(321)$$

\begin{lemma}
For any packed word $\sigma$, we put $\iota(\sigma)=\min\{i\mid \sigma(i)=1\}$.
If $\sigma\leq \tau$, then $\iota(\sigma)=\iota(\tau)$.
\end{lemma}

\begin{proof}
We put $i=\iota(\tau)$. For any $j$, $\tau(j)\geq \tau(i)$, so $\sigma(j)\geq \sigma(i)$ as $\sigma\leq \tau$. So $\sigma(i)=1$, 
and by definition $\iota(\sigma)\leq i$. Let us assume that $j<i$. By definition of $\iota(\tau)$, $\tau(j)>\tau(i)$. As $\sigma\leq \tau$, $\sigma(j)>\sigma(i)$, 
so $\sigma(j)\neq 1$, and $\iota(\sigma)\neq j$. So $\iota(\sigma)=i$.
\end{proof}

\begin{prop}
The map $\varrho:\Surj\longrightarrow \TSchroder$ is a morphism of posets: for any packed words $\sigma,\tau$,
$$\sigma\leq  \tau \Longrightarrow \varrho(\sigma)\leq \varrho(\tau).$$
We define a map $\omega:\TSchroder\longrightarrow\Surj$ by:
\begin{itemize}
\item $\omega(\bun)=1$,
\item $\omega(t_1\vee \ldots \vee t_k)=(\omega(t_1)[1] )1\ldots 1 (\omega(t_k)[1])$.
\end{itemize}
Then $\varrho\circ \omega=Id_{\TSchroder}$, and $\omega$ is a morphism of posets: for any $t,t' \in \TSchroder$,
$$t\leq t'\Longrightarrow \omega(t)\leq \omega(t').$$
\end{prop}

\begin{proof}
Let us prove that $\varrho$ is a morphism. Let $\sigma,\tau$ be two packed words, such that $\sigma\leq \tau$; let us prove that $\varrho(\sigma)\leq \varrho(\tau)$. 
We proceed by induction on the common length $n$ of $\sigma$ and $\tau$. If $n=0$ or $1$, the result is obvious.
Let us assume the result at all rank $<n$. As $\iota(\sigma)=\iota(\tau)$, we can write $\sigma=\sigma'1\sigma''$ and $\tau=\tau'1\tau''$,
where $\sigma'$ and $\tau'$ have the same length and do not contain any $1$. 
By restriction, $Pack(\sigma')\leq Pack(\tau')$ and $Pack(\sigma'')\leq Pack(\tau'')$.  By the induction hypothesis, $s_0=\varrho(\sigma')\leq \varrho(\tau')=t_0$ 
and $s_1\vee \ldots \vee s_k=\varrho(\sigma'')\leq \varrho(\tau'')=t_1\vee \ldots \vee t_l$. Then:
$$\varrho(\sigma)=s_0\vee s_1\vee \ldots \vee \ldots s_k \leq t_0\vee t_1\vee \ldots \vee t_l=\varrho(\tau).$$

Let us now prove that $\omega$ is a morphism. Let $t$, $t' \in \TSchroder$, such that $t \leq t'$. By transitivity, we can assume that there exists $e \in R(t')$,
such that $t=t'_{\mid e}$. Let us prove that $\omega(t) \leq \omega(t')$. We proceed by induction on the common degree $n$ of $t$ and $t'$.
The result is obvious if $n=0$ or $1$. Let us assume the result at all ranks $<n$. We put
$t'=t'_1\vee \ldots \vee t'_k$. If $e$ is an edge of $t'_i$, then $t=t'_1\vee \ldots \vee (t'_i)_{\mid I} \vee \ldots \vee t'_k$.
We put $\sigma'_j=\omega(t'_j)$ and $\sigma_j=\omega(t_j)$ for all $j$. If $j\neq i$, $\sigma'_j=\sigma_j$; 
by the induction hypothesis, $\sigma_i \leq \sigma'_i$. Then:
\begin{align*}
\omega(t)&=(\sigma_1[1])1\ldots 1(\sigma_i[1])1\ldots 1(\sigma_k[1])\\
&\leq (\sigma_1[1])1\ldots 1(\sigma'_i[1])1\ldots 1(\sigma_k[1])=\omega(t').
\end{align*}

If $e$ is the edge relation the root of $t$ to the root of $t'_k$, putting $t=t_1\vee \ldots \vee t_k \vee \ldots \vee t_l$,
then $t'_i=t_i$ if $i<k$ and $t'_k=t_k\vee \ldots \vee t_l$. Putting $\sigma_i=\omega(t_i)$, we obtain:
\begin{align*}
\omega(t)&=(\sigma_1[1])1\ldots1 (\sigma_k[1])1\ldots 1(\sigma_l[1]),\\
\omega(t')=(\sigma_1[1])1\ldots 1(\sigma_k[2])2\ldots 2 (\sigma_l[2]). 
\end{align*}
It is not difficult to prove that $\omega(t)\leq \omega(t')$. \end{proof} 

{\bf Remark.} There are similar results for decorated packed words, replacing $NSh(1)$ and $NQSh(1)$ by $NSh({\N^\ast}^n)$ and $NQSh(\N^*)$.\\

{\bf Examples.}
\begin{align*}
\omega(\bdeux)&=(1),&
\omega(\btroisun)&=(21),&
\omega(\btroisdeux)&=(12),&
\omega(\schtrois)&=(11),\\
\omega(\bquatreun)&=(321),&
\omega(\bquatredeux)&=(231),&
\omega(\bquatretrois)&=(132),&
\omega(\bquatrequatre)&=(123),\\
\omega(\bquatrecinq)&=(212),&
\omega(\schquatreun)&=(221),&
\omega(\schquatredeux)&=(211),&
\omega(\schquatretrois)&=(121),\\
\omega(\schquatrequatre)&=(112),&
\omega(\schquatrecinq)&=(122),&
\omega(\schquatresix)&=(111).
\end{align*}

\begin{prop}
The map $\Psi$ is a bidendriform bialgebra morphism from $(\Sh,\prec,\succ,\Delta_\prec,\Delta_\succ)$ to $({\QSh},\preceq,\succ,
\Delta_\prec,\Delta_\succ)$. Moreover, the following diagram commutes:
$$\xymatrix{\PBT\ar[r]^{\psi}\ar[d]_{\Omega'}&\ST\ar[d]^{\Omega}\\
\Sh\ar[r]_\Psi&{\QSh}}$$
\end{prop}

\begin{proof} Let $\sigma$ be a packed word. We put:
\begin{align*}
A&=\{(k,\tau)\mid \tau\leq \sigma, k\in [\max(\tau)]\},\\
B&=\{(k,\tau',\tau'')\mid k\in [\max(\sigma)],\tau'\leq \sigma_{\mid [k]}, \tau''\leq Pack(\sigma_{\mid[\max(\sigma)]\setminus [k]}).
\end{align*}
 As $\Psi$ is a coalgebra morphism,
\begin{align*}
\Delta\circ \Psi(\sigma)&=\sum_{\tau\leq \sigma} \sum_{k=0}^{\max(\tau)} \tau_{\mid [k]}\otimes Pack(\tau_{\mid [\max(\tau)]\setminus[k]})\\
&=\sum_{(k,\tau) \in A} \tau_{\mid [k]}\otimes Pack(\tau_{\mid [\max(\tau)]\setminus[k]})\\
=(\Psi \otimes \Psi)\circ \Delta(\sigma)&=\sum_{k=0}^{\max(\sigma)}
\sum_{\substack{\tau'\leq \sigma_{\mid [k]}\\\tau''\leq Pack(\sigma_{\mid[\max(\sigma)]\setminus [k]})}} \tau'\otimes \tau''\\
&=\sum_{(l,\tau',\tau'') \in B} \tau'\otimes \tau''.
\end{align*}
Hence, there exists a bijection $F:A\longrightarrow B$, such that, if $F(k,\tau)=(l,\tau',\tau'')$, then:
\begin{itemize}
\item $\tau'=\tau_{\mid[k]}$ and $\tau''=Pack(\tau_{\mid [\max(\tau)]\setminus [k]})$;
\item $l$ is the unique integer such that $\tau'\leq \sigma_{\mid [l]}$.
\end{itemize}
If $k\geq \tau(1)$, then the first letter of $\tau$ appears in $\tau_{\mid [k]}$, so the first letter of $\sigma$ appears also in $\sigma_{\mid [l]}$.
Consequently $l\geq \sigma(1)$. Similarly, if $l\geq \sigma(1)$, then $k\geq \tau(1)$. We obtain:
\begin{align*}
\Delta_\prec \circ \Psi(\sigma)&=\\
&=\sum_{(k,\tau) \in A, k\geq \tau(1)} \tau_{\mid [k]}\otimes Pack(\tau_{\mid [\max(\tau)]\setminus[k]})\\
&=\sum_{(l,\tau',\tau'') \in B, l\geq \sigma(1)} \tau'\otimes \tau''\\
&=(\Psi \otimes \Psi)\circ \Delta_\prec(\sigma)
\end{align*}
So $\Psi$ is a morphism of dendriform coalgebras.\\

Let $\sigma,\tau$ be two permutations. We put:
\begin{align*}
C&=\{(\alpha,\zeta)\mid \alpha \in Sh(\max(\sigma),\max(\tau)), \zeta\leq \alpha \circ (\sigma\otimes \tau)\},\\
D&=\{(\beta,\sigma',\tau')\mid \sigma'\leq \sigma, \tau'\leq \tau, \beta \in QSh(\max(\sigma'),\max(\tau'))\},
\end{align*}
Then:
\begin{align*}
\Psi(\sigma\shuffle \tau)&=\sum_{\alpha \in Sh(\max(\sigma),\max(\tau))} \sum_{\zeta\leq \alpha\circ(\sigma\otimes \tau)} \zeta\\
&=\sum_{(\alpha,\zeta)\in C} \zeta\\
=\Psi(\sigma) \qshuffle \Psi(\tau)&=\sum_{\substack{\sigma'\leq \sigma\\\tau'\leq \tau}} \sum_{\beta\in QSh(\max(\sigma'),\max(\tau'))} 
\beta \circ (\sigma'\otimes \tau')\\
&=\sum_{(\beta,\sigma',\tau') \in D} \beta\circ (\sigma'\otimes \tau').
\end{align*}
Hence, there exists a bijection $G:D\longrightarrow C$, such that if $G(\beta,\sigma',\tau')=(\alpha,\zeta)$, then:
\begin{enumerate}
\item $\zeta=\beta\circ (\sigma'\otimes \tau')$;
\item $\alpha$ is the unique  $(\max(\sigma),max(\tau))$-shuffle such that $\zeta\leq \alpha\circ(\sigma\otimes \tau)$.
\end{enumerate} 
Let us assume that $\alpha(1)=1$, and let us prove that $\beta(1)=1$. 
Denoting by $k$ the length of $\sigma$, $1$ appears in the $k$ first letters of $\zeta'=\alpha \circ (\sigma\otimes \tau)$. Let $i \in [k]$, such that $\zeta'(i)=1$.
For any $j$, $\zeta'(i)\leq \zeta'(j)$. As $\zeta\leq \zeta'$, $\zeta(i)\leq \zeta(j)$, so $\zeta(i)=1$: $1$ appears among the $k$ first letters of $\zeta$, so $\beta(1)=1$.

Let us assume that $\alpha(1)\neq 1$. Then $1$ does not appear in the first $k$ letters of $\zeta'$. Let $j>k$, such that $\zeta'(j)=1$.
For all $i\in[k]$, $\zeta'(i)>\zeta'(j)$ and $i<j$. As $\zeta\leq \zeta'$, $\zeta(i)>\zeta(j)$, so $\zeta(i)\neq 1$: 
$1$ does not appear among the first $k$ letters of $\zeta$, so $\beta(1)\neq 1$. Finally, $\alpha(1)=1$ if, and only if, $\beta(1)=1$. Hence:
\begin{align*}
\Psi(\sigma\prec \tau)&=\sum_{(\alpha,\zeta)\in C, \alpha(1)=1} \zeta
=\sum_{(\beta,\sigma',\tau') \in D,\beta(1)=1} \beta \circ (\sigma'\otimes \tau')=\Psi(\sigma) \preceq \Psi(\tau).
\end{align*}

By composition, $\Omega\circ \psi$ and $\Psi \circ \Omega$ are both noncommutative shuffle algebra morphisms, sending $\bdeux$ to $(1)$, so, since $\PBT$ is a free NSh algebra, they are equal.
\end{proof}

\bibliographystyle{amsplain}
\bibliography{biblio}

\end{document}